\documentclass[12pt]{amsart}
\usepackage[margin=1in]{geometry}

\usepackage{url,amsmath,amssymb,mathrsfs}
\usepackage{enumitem}
\usepackage{commath}
\usepackage{physics}
\usepackage{nicefrac}
\usepackage{amsthm}
\usepackage{xfrac}
\usepackage[all]{xy}
\usepackage{graphicx}
\graphicspath{ {./images/} }

\newtheorem{theorem}{Theorem}[section]
\newtheorem{lemma}[theorem]{Lemma}

\newtheorem{corollary}[theorem]{Corollary}

\theoremstyle{definition}

\newtheorem{definition}[theorem]{Definition}
\newtheorem{example}[theorem]{Example}

\newcommand{\bZ}{\mathbb Z}
\newcommand{\bQ}{\mathbb Q}
\newcommand{\bN}{\mathbb N}

\newcommand{\Rbar}{\overline{R}}

\newcommand{\Inv}{\textup{Inv}}

\newcommand{\Cl}{\textup{Cl}}

\newcommand{\modulo}[1]{\textup{ (mod }#1)}

\newcommand{\cO}{\mathcal{O}}
\newcommand{\Irr}{\textup{Irr}}
\newcommand{\term}[1]{\textbf{\textup{#1}}}

\newcommand{\quot}[2]{\large\sfrac{#1}{#2}}
\newcommand{\DivCl}{\textup{DivCl}}

\title[Elasticity of Orders and Power Series]{Elasticity in Orders of an Algebraic Number Field with Radical Conductor Ideal and their Rings of Formal Power Series}
\author{Jim Coykendall
\and
Grant Moles}
\date{May 2024}

\begin{document}

\begin{abstract}
    Orders in an algebraic number field form a class of rings which are of special historical interest to the field of factorization theory. One of the primary tools used to study factorization is elasticity - a measure of how badly unique factorization fails in a domain. This paper explores properties of orders in a number field and how they can be used to study elasticity in not only the orders themselves, but also in rings of formal power series over the orders. Of particular interest is the fact, proven here, that power series extensions in finitely many variables over half-factorial rings of algebraic integers must themselves be half-factorial. It is also shown that the HFD property is not preserved in general for power series rings over non-integrally closed orders in a number field.
\end{abstract}

\maketitle

\section{Introduction}

One of the most fundamental theaters for the study of factorization is the class of orders in a ring of algebraic integers. The paper \cite{carlitz} is an early striking example of how rings of integers form a class of domains amenable to much stronger results than can be found in more general classes of domains. The fact that half-factorial domains (HFDs) have an ideal-theoretic characterization for rings of algebraic integers is one reason why results for HFDs are more robust in this arena. For example, for a ring of algebraic integers that is an HFD, it is known that the HFD property is preserved in localizations and polynomial extensions, neither of which is true in more general settings. But even the much more restrictive class of unique factorization domains (UFDs) are famously not preserved in power series extensions (this was first shown by P. Samuel in \cite{s1961}).

Unlike unique factorization domains, half-factorial domains do not have to be integrally closed, and it is shown in \cite{Co1} that if $R$ is an HFD, then ``integrally closed" is a necessary condition for the polynomial extension $R[x]$ to have the half-factorial property (but not sufficient as can be obtained from the results in, for example, \cite{zaks2} and \cite{GHS1996} in tandem). For Noetherian domains, a complete characterization of polynomial HFDs can be found in \cite{Co1}.

Given the above remarks and the general misbehavior of power series extensions, it is perhaps surprising that ``integrally closed" is not a necessary condition for the domain $R[[x]]$ to be an HFD. It was first shown in the primarily expository paper \cite{Co2005} (Theorem 6.2 and Corollary 6.3) that there are HFDs, $R$, that are not integrally closed despite the fact that $R[[x]]$ is an HFD. A concrete example of this is the order $R:=\mathbb{Z}[\sqrt{-3}]$; for this domain, we have that $R$ and $R[[x]]$ are HFDs, but as $R$ is not integrally closed, $R[x]$ is not an HFD.

Additionally, in the paper \cite{MO1}, it was shown that if $R[[x]]$ is an HFD, it is not true in general that $R[[x,y]]$ is an HFD. This may be considered a companion to the well-known question as to the preservation of the UFD property in power series extensions. As was noted earlier, it is possible for the power series ring $R[[x]]$ to fail to have unique factorization even if $R$ is a UFD, and it is natural to ask if $R[[x]]$ being a UFD implies that $R[[x,y]]$ is a UFD. This question has been open for quite some time, and this makes the results of \cite{MO1} intriguing.

In the papers \cite{halter-koch}, \cite{coykendallhfd}, \cite{Co2}, and more recently \cite{bcmm2025} and \cite{rago}, factorization behavior (especially the half-factorial property) have been investigated from a number of points of view, and this work can be considered a marriage of studying the orders in rings of integers with special attention paid to ring extensions (in particular power series extensions) of these orders. We give a brief overview of the structure of this paper.

In the second section, we develop some closely related but distinct properties (associated, ideal-preserving, and locally associated) that subrings of a commutative ring might possess that are germane to our study of preservation of factorization in these underrings. We study the interplay of these properties and their relation to the unit group and class group of the domains with an eye toward factorization. 

In the third section, we use the tools developed in the second section and apply these concepts to glean insights to factorization properties and elasticity of an order $R$ and its integral closure (full ring of integers) $\overline{R}$.

In sections four and five, we leverage the previous results to the power series ring $R[[x]]$, where $R$ is an order in the ring of integers $\overline{R}$. In particular, it is shown that if $R$ is an associated order with radical conductor ideal and full ring of integers $\overline{R}$, then the elasticities of $R[[x]]$ and $\overline{R}[[x]]$ are equal, and so we can conclude that if $R$ is an order with radical conductor ideal, then $R$ is an HFD if and only if $R[[x]]$ is an HFD; as a corollary, we note that this is universally true in the case of quadratic orders.

In the sixth and final section, we provide some specialty results in the case in which we have a non-radical conductor and construct some examples to illustrate some intricacies and hazards moving forward in this case.

\section{Associated Orders}
    When considering the factorization of elements in an order (or its ring of power series), certain types of subrings will be very useful to consider. The first will look familiar from the main result of \cite{halter-koch} and describes a subring whose multiplicative structure is, in a sense, ``close" to that of the entire ring.
    \begin{definition}
        \label{associated subring}
        Let $T$ be a commutative ring with identity and $R$ a subring of $T$ (not necessarily with identity). We say that $R$ is an \term{associated subring} of $T$ if $T=R\cdot U(T)$; that is, if any element $t\in T$ can be expressed as $t=ru$ for some $r\in R$ and $u\in U(T)$. Equivalently, for any element $t\in T$, there exists $u\in U(\Rbar)$ such that $tu\in R$.
    \end{definition}
    Related to this type of subring (though this relation may not be immediately apparent) are the following.
    \begin{definition}
        \label{ideal-preserving subring}
        Let $T$ be a commutative ring and $R$ a subring of $T$ (neither assumed to have identity). We say that $R$ is an \term{ideal-preserving subring} of $T$ if, for any $T$-ideals $J_1\not\subseteq J_2$, $R\cap J_1\not\subseteq R\cap J_2$ (equivalently, $R\cap J_1\not\subseteq J_2$). 
    \end{definition}
    \begin{definition}
        \label{locally associated subring}
        Let $T$ be a commutative ring with identity, $R$ a subring of $T$ with identity, and $I=(R:T)$, the conductor ideal from $T$ into $R$. We say that $R$ is a \term{locally associated subring} of $T$ if
        $$\quot{U(T)}{U(R)}\cong \quot{U(\sfrac{T}{I})}{U(\sfrac{R}{I})}.$$
    \end{definition}
    Though such subrings are interesting in their own right and warrant further investigation, we will focus here on the case when $R$ is an order in a number field $K$ and $T=\Rbar=\cO_K$. In this case, we can produce a number of equivalent characterizations of these properties. These characterizations, as well as the properties they allow us to prove, will be useful to us throughout this paper.
    \begin{theorem}
        \label{associated order}
        Let $R$ be an order in a number field $K$. The following are equivalent.
        \begin{enumerate}
            \item $R$ is an associated subring of $\Rbar$.
            \item For any $t\in \Rbar$ and subset $\{u_\alpha\}_{\alpha\in \Gamma}\subseteq U(\Rbar)$ containing a representative from each coset of $\quot{U(\Rbar)}{U(R)}$, there exist $\alpha\in \Gamma$ and $\beta\in I$ such that $u_\alpha(t+\beta)\in R$. That is, in a slight abuse of notation, $\quot{\Rbar}{I}=\quot{R}{I}\cdot \quot{U(\Rbar)}{U(R)}$.
        \end{enumerate}
        For simplicity, we will refer to any such order $R$ as an \term{associated order}.
    \end{theorem}
    \begin{proof}
        First, assume that $R$ is an associated subring of $\Rbar$. Also assume that we have a subset $\{u_\alpha\}_{\alpha\in\Gamma}\subseteq U(\Rbar)$ as described in Condition 2. Then for any $t\in \Rbar$, there exists some $u\in U(\Rbar)$ such that $tu=r\in R$. Now select $\alpha\in\Gamma$ such that $u_\alpha\equiv u\modulo{U(R)}$. Then there exists some $v\in U(R)$ such that $uv=u_\alpha$, so $tu_\alpha=tuv=rv\in R$. Then letting $\beta=0\in I$, we have that $u_\alpha(t+\beta)\in R$, so $1\implies 2$.

        Now assume that Condition 2 holds and let $t\in \Rbar$. For any subset $\{u_\alpha\}_{\alpha\in\Gamma}\subseteq U(\Rbar)$ as above, we can select $u_\alpha$ and $\beta\in I$ such that $u_\alpha(t+\beta)=r\in R$. Then $u_\alpha t=r-u_\alpha \beta$. Since $\beta\in I$, $u_\alpha\beta\in R$, so $u_\alpha t\in R$. Then $R$ must be an associated subring of $\Rbar$, so $2\implies 1$.
    \end{proof}

    \begin{theorem}
        \label{ideal-preserving order}
        Let $R$ be an order in a number field $K$ with conductor ideal $I$. The following are equivalent.
        \begin{enumerate}
            \item $R$ is an ideal-preserving subring of $\Rbar$.
            \item For any prime $\Rbar$-ideals $P_1\neq P_2$ dividing $I$, $R\cap P_1\not\subseteq P_2$ and $R\cap P_1\nsubseteq P_1^2$.
            \item If $I=P_1^{a_1}\dots P_k^{a_k}$ is the factorization of $I$ into prime $\Rbar$-ideals, then $R\cap P_i\nsubseteq P_i^2$ for $1\leq i\leq k$ and $$\quot{R}{I}\cong \prod_{i=1}^k\quot{R+P_i^{a_i}}{P_i^{a_i}}\cong \prod_{i=1}^k \quot{R}{R\cap P_i^{a_i}}.$$
        \end{enumerate}
        For simplicity, we will refer to any such order as an \term{ideal-preserving order}.
    \end{theorem}
    \begin{proof}
        First, note that for each $1\leq i\leq k$, $$\quot{R}{R\cap P_i^{a_i}}\cong \quot{R+P_i^{a_i}}{P_i^{a_i}}.$$ This can be seen quite easily by considering the map $\phi:\quot{R}{R\cap P_i^{a_i}}\to\quot{R+P_i^{a_i}}{P_i^{a_i}}$ defined by $\phi(r+R\cap P_i^{a_i})=r+P_i^{a_i}$, which is an isomorphism (even without assuming that $R$ is ideal-preserving). Then the second isomorphism in Condition 3 will always hold.

        Now assume that $R$ is an ideal-preserving subring of $\Rbar$. Note that for any prime ideal $P_i|I$, $P_i\nsubseteq P_i^2$, so $R\cap P_i\nsubseteq R\cap P_i^2$ by definition of an ideal-preserving subring. Then to show that Condition 3 holds, we only need to show that $$\quot{R}{I}\cong \prod_{i=1}^k\quot{R+P_i^{a_i}}{P_i^{a_i}}.$$ First, note that by the Chinese Remainder Theorem, $$\quot{\Rbar}{I}\cong \prod_{i=1}^k\quot{\Rbar}{P_i^{a_i}}.$$ Let $\pi:\quot{\Rbar}{I}\to \prod_{i=1}^k\quot{\Rbar}{P_i^{a_i}}$ denote the canonical isomorphism which projects $\Rbar$ onto each coordinate, and denote by $\tau:\quot{R}{I}\to \prod_{i=1}^k\quot{R+P_i^{a_i}}{P_i^{a_i}}$ the restriction of $\pi$ to $\quot{R}{I}$ (it should be clear that this codomain for $\tau$ is appropriate). As a restriction of an injective homomorphism, $\tau$ will automatically be an injective homomorphism as well. We need to show that $\tau$ is surjective. To do so, we will show that for each $1\leq i\leq k$, there exists some $\alpha_i\in R$ such that $\alpha_i\equiv 1\modulo{P_i^{a_i}}$ and $\alpha_i\equiv 0\modulo{P_j^{a_j}}$ for each $j\neq i$.

        For any $1\leq i\leq k$, note that $IP_i^{-a_i}=\prod_{j\neq i}P_j^{a_j}\nsubseteq P_i.$ Then since $R$ is an ideal-preserving subring of $\Rbar$, $R\cap IP_i^{-a_i}\nsubseteq P_i$. Then let $x_i\in R\cap IP_i^{-a_i}\backslash P_i$. Since $x_i\notin P_i$, $x_i+P_i^{a_i}\in U(R+P_i^{a_i})$; let $y_i\in R+P_i^{a_i}$ such that $y_i+P_i^{a_i}=(x+P_i^{a_i})^{-1}$. Without loss of generality, since $y_i$ only needs to be in a specific ideal class of $R+P_i^{a_i}$ modulo $P_i^{a_i}$, assume that $y_i\in R$. Then let $\alpha_i=x_iy_i$. Since $x_i\in P_j^{a_j}$ for every $j\neq i$, so is $\alpha_i$. Since $x_i,y_i\in R$, so is $\alpha_i$. Finally, since $y_i+P_i^{a_i}=(x+P_i^{a_i})^{-1}$, $\alpha_i=x_iy_i\equiv 1\modulo{P_i^{a_i}}$. Then $\tau(\alpha_i+I)$ is zero in every coordinate aside from the $i^{th}$, in which it is $1+P_i^{a_i}$. Since it is possible to construct such $\alpha_i$ for each $1\leq i\leq k$, it follows that $\tau$ is surjective. Thus, $\tau$ is an isomorphism and $1\implies 3$.

        Now assume that Condition 3 holds. To show that Condition 2 holds, it will suffice to show that for prime $\Rbar$-ideals $P_1\neq P_2$ dividing $I$, $R\cap P_1\nsubseteq P_2$. Let $I=P_1^{a_1}P_2^{a_2}\dots P_k^{a_k}$ and $\tau$ be the isomorphism described previously. Since the isomorphisms in Condition 3 hold, there must exist some $\alpha_1\in R$ such that $\tau(\alpha_1+I)$ is 0 in the first coordinate and 1 in the second. Then $\alpha\in R\cap P_1\backslash P_2$, meaning that $R\cap P_1\nsubseteq P_2$. Then $3\implies 2$.

        Now assume that Condition 2 holds. First, we will show that for any prime $\Rbar$-ideals $P_1\neq P_2$, $R\cap P_1\nsubseteq P_2$ and $R\cap P_1\nsubseteq P_1^2$. If both $P_1$ and $P_2$ divide $I$, then we are done. Now assume that $P_2\nmid I$ and let $\beta\in P_1\backslash P_2$, $\gamma\in I\backslash P_2$, and $\alpha=\beta\gamma$. Note that $\alpha\in I\subseteq R$ and $\alpha\notin P_2$. Then $\alpha\in R\cap P_1\backslash P_2$, so $R\cap P_1\nsubseteq P_2$. Now assume that $P_1$ does not divide $I$. Letting $\beta\in P_1\backslash P_1^2$, $\gamma\in I\backslash P_1$, and $\alpha=\beta\gamma$, we have that $\alpha\in R\cap P_1\backslash P_1^2$, so $R\cap P_1\nsubseteq P_1^2$. Now if $P_1\nmid I$ and $P_2|I$, let $\alpha\in P_1$ and $\beta\in I$ such that $\alpha+\beta=1$. Then $\alpha=1-\beta\in R\cap P_1$, and since $\beta\in P_2$, $\alpha=1-\beta\notin P_2$. Then $R\cap P_1\nsubseteq P_2$. This covers all cases, so for any prime $\Rbar$-ideals $P_1\neq P_2$, $R\cap P_1\nsubseteq P_2$ and $R\cap P_1\nsubseteq P_1^2$.

        Now let $J_1\nsubseteq J_2$ be $\Rbar$-ideals. Let $J_1=P_1^{a_1}\dots P_r^{a_r}$ and $J_2=P_1^{b_1}\dots P_r^{b_r}$, with each $P_i$ a prime $\Rbar$ ideal and $a_i,b_i\in\bN_0$ for $1\leq i\leq r$. Since $J_2\nmid J_1$, there must exist some $1\leq i\leq k$ such that $b_i> a_i$; without loss of generality, assume $b_1> a_1$. Then let $\beta_1\in R\cap P_1\backslash P_1^2$, and for each $2\leq i\leq k$, let $\beta_i\in R\cap P_i\backslash P_1$. Then $\alpha=\beta_1^{a_1}\dots \beta_r^{a_r}\in R\cap J_1$. However, since $\alpha\in P_1^{a_1}\backslash P_1^{a_1+1}$ and $P_1^{b_1}\subseteq P_1^{a_1}+1$, $\alpha\notin J_2$. Then $R\cap J_1\nsubseteq J_2$, so $R$ is an ideal-preserving subring of $\Rbar$. Thus, $2\implies 1$.
    \end{proof}
    For similar characterizations of orders which are locally associated subrings of their integral closures, we require the following result from \cite{neukirch}.
    \begin{lemma}
        \label{exact sequence}
        Let $R$ be an order in a number field $K$ with conductor ideal $I$. Then there is an exact sequence
        $$1\to U(R)\to U(\Rbar)\times U(\quot{R}{I})\to U(\quot{\Rbar}{I})\to\normalsize\Cl(R)\to \Cl(\Rbar)\to 1$$
        Thus, the class numbers $\abs{\Cl(R)}$ and $\abs{\Cl(\Rbar)}$ are related as follows:
        $$\abs{\Cl(R)}=\abs{\Cl(\Rbar)}\frac{\abs{U(\quot{\Rbar}{I})}}{\abs{U(\quot{R}{I})}\cdot\abs{\quot{U(\Rbar)}{U(R)}}}.$$
    \end{lemma}
    \begin{theorem}
        \label{locally associated orders}
        Let $R$ be an order in a number field $K$ with conductor ideal $I$. The following are equivalent.
        \begin{enumerate}
            \item $R$ is a locally associated subring of $\Rbar$.
            \item Every coset in $\quot{U(\sfrac{\Rbar}{I})}{U(\sfrac{R}{I})}$ contains a unit in $\Rbar$; that is, for any $t+I\in U(\quot{\Rbar}{I})$, there exists some $r+I\in U(\quot{R}{I})$ and $\beta\in I$ such that $tr+\beta\in U(\Rbar)$.
            \item If $t\in \Rbar$ is relatively prime to $I$, i.e. $t\Rbar+I=\Rbar$, then there exists $r\in R$ relatively prime to $I$, i.e. $rR+I=R$, and $u\in U(\Rbar)$ such that $t=ru$.
            \item $\abs{\quot{U(\Rbar)}{U(R)}}=\frac{\abs{U(\quot{\Rbar}{I})}}{\abs{U(\quot{R}{I})}}$.
            \item $\abs{\Cl(\Rbar)}=\abs{\Cl(R)}$.
            \item $\Cl(\Rbar)\cong \Cl(R)$.
        \end{enumerate}
        For simplicity, we will refer to any such order as a \term{locally associated order}.
    \end{theorem}
    \begin{proof}
        First, note that the groups $\Cl(\Rbar)$, $\Cl(R)$, $\quot{U(\Rbar)}{U(R)}$, $U(\quot{\Rbar}{I})$, and $U(\quot{R}{I})$ are finite. Then $1\iff 4$ and $5\iff 6$ are trivial; $4\iff 5$ follows immediately from the lemma.

        Now assume that $R$ is a locally associated subring of $\Rbar$, and let $(t+I)U(\quot{R}{I})\in \quot{U(\sfrac{\Rbar}{I})}{U(\sfrac{R}{I})}$ for some $t\in \Rbar$. Let $\phi:\quot{U(\Rbar)}{U(R)}\to \quot{U(\sfrac{\Rbar}{I})}{U(\sfrac{R}{I})}$ be the homomorphism defined by $\phi(u\cdot U(R))=(u+I)U(\quot{R}{I})$. It is straightforward to show that this is an injective homomorphism; since it maps a finite group injectively into a finite group of the same order, $\phi$ must be an isomorphism. Then there must be some $u\in U(\Rbar)$ such that $\phi(u\cdot U(R))=(u+I)U(\quot{R}{I})=(t+I)U(\quot{R}{I})$. Thus, there must be some $r+I\in U(\quot{R}{I})$ such that $u+I=(t+I)(r+I)=tr+I$, so there is some $\beta\in I$ such that $u=tr+\beta\in U(\Rbar)$. Then $1\implies 2$.

        Now assume that Condition 2 holds, and let $t\in \Rbar$ be relatively prime to $I$. Then $t+I\in U(\quot{R}{I})$, so there must exist $r+I\in U(\quot{R}{I})$ and $\beta\in I$ such that $tr+\beta=u\in U(\Rbar)$. Let $s\in R$ be such that $s+I=(r+I)^{-1}\in U(\quot{R}{I})$ and note that $s$ is relatively prime to $I$ in $R$. Then for some $\gamma\in I$, $t+\gamma=su\implies t=(s-u^{-1}\gamma)u$. Since $s$ is relatively prime to $I$ in $R$ and $u^{-1}\gamma\in I$, then $s-u^{-1}\gamma$ is also relatively prime to $I$ in $R$. Then $2\implies 3$.

        Finally, assume that Condition 3 holds. As discussed above, the map $\phi:\quot{U(\Rbar)}{U(R)}\to \quot{U(\sfrac{\Rbar}{I})}{U(\sfrac{R}{I})}$ such that $\phi(u\cdot U(R))=(u+I)U(\quot{R}{I})$ is always an injective homomorphism. We need to show that this map is surjective. Then let $t+I\in U(\quot{\Rbar}{I})$, i.e. $t\in\Rbar$ is relatively prime to $I$. Then there must exist some $r\in R$ relatively prime to $I$ and $u\in U(\Rbar)$ such that $t=ru$. Letting $s+I=(r+I)^{-1}\in U(\quot{R}{I})$, we have that $(t+I)(s+I)=u+I$. Then $\phi(u\cdot U(R))=(u+I)U(\quot{R}{I})=(t+I)U(\quot{R}{I})$, so $\phi$ is a surjective map. Thus, $\phi$ is an isomorphism and $3\implies 1$.
    \end{proof}
    Before exploring some helpful properties of these types of orders, we should first consider the relationships among these properties.
    \begin{theorem}
        \label{ao implies ipo and lao}
        Let $R$ be an order in a number field $K$. If $R$ is an associated order, then $R$ is both ideal-preserving and locally associated.
    \end{theorem}
    \begin{proof}
        Suppose that $R$ is an associated order, and let $J_1\nsubseteq J_2$ be $\Rbar$-ideals. Select any $\alpha\in J_1\backslash J_2$ and let $u\in U(\Rbar)$ be such that $\alpha u\in R$. Since $\alpha\in J_1$, $\alpha u\in J_1$ as well. if $\alpha u=\beta\in J_2$, then $\alpha=u^{-1}\beta\in J_2$, a contradiction. Then $\alpha u\in R\cap J_1\backslash J_2$, so $R\cap J_1\nsubseteq J_2$. Thus, $R$ is an ideal-preserving order.

        Now let $t\in \Rbar$ be relatively prime to $I$. Since $R$ is an associated order, there exists some $r\in R$ and $u\in U(\Rbar)$ such that $t=ru$. Note that $r+I=u^{-1}t+I\in \quot{R}{I}\cap U(\quot{\Rbar}{I})$. Since $\Rbar$ is integral over $R$ (and thus $\quot{\Rbar}{I}$ is integral over $\quot{R}{I}$), this means that $r+I\in U(\quot{R}{I})$, i.e. $r$ is relatively prime to $I$ in $R$. Then by the third equivalent condition in Theorem \ref{locally associated orders}, $R$ is a locally associated order.
    \end{proof}
    It is worth noting that there exist ideal-preserving orders which are neither associated nor locally associated, and there exist locally associated orders which are neither associated nor ideal-preserving. The following examples illustrate these points.

    \begin{example}
        Let $R=\bZ[5\sqrt{2}]$; then $\Rbar=\bZ[\sqrt{2}]$, $I=(R:\Rbar)=5\Rbar$, and $U(\Rbar)=\{\pm(1+\sqrt{2})^k|k\in\bZ\}$. Note that 5 is an inert prime in $\Rbar$, i.e. $I$ is a prime $\Rbar$-ideal. Then $R\cap I=I\nsubseteq I^2$, so Condition 2 from Theorem 2.5 tells us that $R$ is an ideal-preserving order. On the other hand, note that $(1+\sqrt{2})^2=3+2\sqrt{2}\notin R$ and $(1+\sqrt{2})^3=7+5\sqrt{2}\in R$. Then it is easy to see that $\abs{\quot{U(\Rbar)}{U(R)}}=3$. However, $\abs{U(\quot{\Rbar}{I})}=\abs{\quot{\Rbar}{I}}-1=24$ and $\abs{U(\quot{R}{I})}=\abs{\quot{R}{I}}-1=4$, so $$\abs{\quot{U(\Rbar)}{U(R)}}=3\neq 6=\frac{\abs{U(\quot{\Rbar}{I})}}{\abs{U(\quot{R}{I})}}.$$ Then Condition 4 from Theorem \ref{locally associated orders} tells us that $R$ is not locally associated (and thus not associated). 
    \end{example}

    \begin{example}
        Let $R=\bZ[2\sqrt{2}]$; then $\Rbar=\bZ[\sqrt{2}]$, $I=(R:\Rbar)=2\Rbar$, and $U(\Rbar)=\{\pm(1+\sqrt{2})^k|k\in\bZ\}$. The prime factorization of $I$ in $\Rbar$ is $I=(\sqrt{2})^2$, and $R\cap(\sqrt{2})=I=R\cap (\sqrt{2})^2$. Then $I$ is not ideal-preserving (and thus not associated). Now note that $U(\quot{\Rbar}{I})=\{1+I,(1+\sqrt{2})+I\}$. Since both $1,1+\sqrt{2}\in U(\Rbar)$, $R$ is a locally associated order by Condition 2 of Theorem \ref{locally associated orders}.
    \end{example}
    
    Among other things, these types of orders and their properties will be useful for drawing conclusions about the relationship between an order $R$, its integral closure $\Rbar$, and intermediate orders $T$ such that $R\subseteq T\subseteq \Rbar$. First, we have the following theorem which tells us that these properties are ``inherited" by such intermediate orders. For part of this proof, we will need the following lemmas.

    \begin{lemma}
        \label{rel prime is invertible}
        \cite{conrad} Let $R$ be an order in a number field $K$ with conductor ideal $I$. Then any $R$-ideal which is relatively prime to $I$ is invertible. That is, if $J$ is an ideal in $R$ such that $J+I=R$, then $JJ^{-1}=R$, with $J^{-1}=\{\alpha\in K|\alpha J\subseteq R\}$.
    \end{lemma}

    \begin{lemma}
        \label{rel prime in ideal class}
        \cite{conrad} Let $R$ be an order in a number field $K$ and $J$ and ideal in $R$. Then every ideal class in $\Cl(R)$ contains a representative which is an integral ideal of $R$ that is relatively prime to $J$. That is, for any $A\in \Inv(R)$, there exists $\alpha\in K$ such that $\alpha A\subseteq R$ and $\alpha A+J=R$.
    \end{lemma}

    \begin{lemma}
        \label{surjective class map}
        Let $R$ be an order in a number field $K$ with conductor ideal $I$, and let $T$ be an intermediate order, $R\subseteq T\subseteq \Rbar$. Then the mapping $\phi:\Cl(R)\to \Cl(T)$ such that $\phi([J])=[JT]$ is a surjective homomorphism.
    \end{lemma}
    \begin{proof}
        First, note that $\phi$ is well-defined. If $[J_1]=[J_2]$ for two invertible fractional $R$-ideals $J_1$ and $J_2$, then $J_1=\alpha J_2$ for some $\alpha\in K$. Thus, $J_1T=\alpha J_2T$, so $\phi([J_1])=[J_1T]=[J_2T]=\phi([J_2])$. Furthermore, $\phi$ is a homomorphism, since for any invertible fractional $R$-ideals $J_1$ and $J_2$, $\phi([J_1J_2])=[J_1J_2T]=[J_1T][J_2T]=\phi([J_1])\phi([J_2])$. All that remains to show is that $\phi$ is surjective.

        Let $[A]\in \Cl(T)$. Since $I\subseteq T$, we can use Lemma \ref{rel prime in ideal class} to assume without loss of generality that $A$ is an integral ideal of $T$ which is relatively prime to $I$. The let $J=R\cap A$. Since $A$ is relatively prime to $I$, we know that there is some $\alpha\in A$ and $\beta\in I$ such that $\alpha+\beta=1$. Then $\alpha=1-\beta\in R$, so in fact $\alpha\in J$. Then $J$ is relatively prime to $I$ as an $R$-ideal. Lemma \ref{rel prime is invertible} tells us that $J\in \Inv(R)$. Moreover, note the following:
        $$A=AR=A(J+I)=AJ+AI\subseteq JT\subseteq A.$$ Then $A=JT$, so $\phi([J])=[JT]=[A]$. Thus, $\phi$ is a surjective homomorphism.
    \end{proof}

    \begin{theorem}
        \label{inheritance}
        Let $R$ and $T$ be two orders in the same number field $K$, with $R\subseteq T\subseteq \Rbar$. Then:
        \begin{enumerate}
            \item If $R$ is an associated order, then so is $T$.
            \item If $R$ is an ideal-preserving order, then so is $T$.
            \item IF $R$ is a locally associated order, then so is $T$.
        \end{enumerate}
    \end{theorem}
    \begin{proof}
        The first statement here is the easiest to see. If $R$ is an associated order, then for any $\alpha\in\Rbar$, there exists $u\in U(\Rbar)$ such that $u\alpha\in R\subseteq T$. Thus, $T$ is also an associated order.

        Now assume that $R$ is an ideal-preserving order, and let $J_1\nsubseteq J_2$ be $\Rbar$-ideals. Then $R\cap J_1\nsubseteq J_2$, i.e. there is some $r\in R\cap J_1$ which does not lie in $J_2$. Then $r\in T\cap J_1\backslash J_2$, so $T$ is also ideal-preserving.

        Finally, assume that $R$ is a locally associated order. By Theorem \ref{associated order}, we know that $\abs{\Cl(\Rbar)}=\abs{\Cl(R)}$. By Lemma \ref{surjective class map}, $\abs{\Cl(R)}\geq\abs{\Cl(T)}$; by Lemma \ref{exact sequence}, $\abs{\Cl(T)}\geq\abs{\Cl(\Rbar)}$. Then $\abs{\Cl(T)}\geq\abs{\Cl(\Rbar)}=\abs{\Cl(R)}\geq\abs{\Cl(T)}$, so $\abs{\Cl(T)}=\abs{\Cl(\Rbar)}$. Then by Condition 5 in Theorem \ref{locally associated orders}, $T$ is a locally associated order.
    \end{proof}
    The following results demonstrate some of the utility of these properties. Namely, they show how we might determine some of the structure of a particular type of intermediate order.
    \begin{theorem}
        \label{int orders}
        Let $R$ be an ideal-preserving order in a number field $K$ with conductor ideal $I$. Then for any $\Rbar$-ideal $J$, $R+J$ is an order in $K$ with conductor ideal $I+J$. In particular, if $J|I$, then $R+J$ has conductor ideal $J$.
    \end{theorem}
    \begin{proof}
        First, note that $I+J$ is an $\Rbar$-ideal which divides $J$, and $R+J=R+(I+J)$. Then it will suffice to show that for any $\Rbar$-ideal $J$ which divides $I$, the ring $R+J$ has conductor ideal $(R+J:T)=J$.

        Now assume that $J$ is an $\Rbar$-ideal dividing $I$ and consider the order $R+J$ in $K$. First, note that $J\subseteq R+J$ trivially. Then letting $A$ be the conductor ideal of $R+J$, we have that $J\subseteq A$. If $A=J$, then we are done. Otherwise, assume that $J\subsetneq A$. Then $R+J\subseteq R+A\subseteq R+(R+J)=R+J$, so $R+J=R+A$.

        Let $I=P_1^{a_1}\dots P_k^{a_k}$ be the factorization of $I$ into prime $\Rbar$-ideals. Then since $J|I$, $J=P_1^{b_1}\dots P_k^{b_k}$ for some $0\leq b_i\leq a_i$, $1\leq i\leq k$. Since $R$ is an ideal-preserving order, we have that for each $1\leq i\leq k$, $R\cap P_i\not\subseteq R\cap P_i^2$, and for $i\neq j$, $R\cap P_i\nsubseteq R\cap P_j$. Note that each $R\cap P_i$ remains prime in $R$. Then the Prime Avoidance Lemma tells us that for each $1\leq i\leq k$, there exists some $r_i\in R\cap P_i\backslash \left(P_i^2\cup\bigcup_{j\neq i}P_j\right)$. Then letting $\beta=\prod_{i=1}^kr_i^{a_i-b_i}$, we have that $\beta\in R$ and $\beta \Rbar+I=P_1^{a_1-b_1}\dots P_k^{a_k-b_k}=IJ^{-1}$.

        Now note that $\beta(R+A)=\beta(R+J)=\beta R+\beta J$. Since $\beta\in R$, $\beta R\subseteq R$; since $\beta\in IJ^{-1}$, $\beta J\subseteq I$. Then $\beta(R+A)=\beta R+\beta J\subseteq R+I=R$. Then $\beta$ conducts any element of $R+A$ into $R$, i.e. $\beta\in (R:R+A)$. Now since $A\nsubseteq J$ and $R$ is an ideal-preserving order, $R\cap A\nsubseteq J$. Then let $\alpha\in R\cap A\backslash J$. Since $(\alpha)\nsubseteq J$, there must be some $1\leq i\leq k$ such that $P_i^{b_i}\nmid (\alpha)$. Now consider the element $\alpha\beta$. Note that for any $t\in\Rbar$, $\alpha t\in A\subseteq R+A$, so $\alpha\beta t=\beta(\alpha t)\in R$. Then $\alpha\beta\in (R:\Rbar)=I$, i.e. $I|(\alpha\beta)$. However, note that as there is some $1\leq i\leq k$ such that $P_i^{b_i}\nmid (\alpha)$ and the exact power of $P_i$ dividing $(\beta)$ is $P_i^{a_i-b_i}$, $P_i^{a_i}\nmid(\alpha\beta)$, a contradiction. Then the conductor ideal $(R+J:T)$ must be exactly $J$.
    \end{proof}

    \begin{theorem}
        \label{intersect orders}
        Let $R$ be an ideal-preserving order in a number field $K$ with conductor ideal $I$, and let $J_1$ and $J_2$ be $\Rbar$-ideals dividing $I$. Then $(R+J_1)\cap (R+J_2)=R+J_1\cap J_2$.
    \end{theorem}
    \begin{proof}
        For ease of notation, define $R_1:=R+J_1$ and $R_2:=R+J_2$. For $i\in\{1,2\}$, the inclusion $R+J_1\cap J_2\subseteq R_i$ is trivial; then $R+J_1\cap J_2\subseteq R_1\cap R_2$. We need to show the reverse inclusion. To do so, we will first note that $R+J_1\cap J_2$ is an ideal-preserving order with conductor ideal $J_1\cap J_2$, and $R_i=R+J_i=(R+J_1\cap J_2)+J_1$ for $i\in\{1,2\}$. Then it will suffice to show that the result holds when $J_1\cap J_2=I$.

        Assume that $I=J_1\cap J_2$; for now, also assume that $J_1$ and $J_2$ are relatively prime. We need to show that $R_1\cap R_2\subseteq R$. Applying the isomorphism in Condition 3 of Theorem \ref{ideal-preserving order} first to $R$, then to $R_1$ and $R_2$, we get that $\quot{R}{I}\cong\quot{R_1}{J_1}\times\quot{R_2}{J_2}$. Then letting $t\in R_1\cap R_2$ and $\tau:\quot{R}{I}\to \quot{R_1}{J_1}\times \quot{R_2}{J_2}$ the inclusion isomorphism between these rings, there must be some $r\in I$ such that $\tau(r+I)=(r+J_1,r+J_2)=(t+J_1,t+J_2)$. Then $r-t\in J_1\cap J_2=I$, so $t=r+\beta\in R$ for some $\beta\in I$. Then $R_1\cap R_2\subseteq R.$

        All that remains to show is the case when $J_1\cap J_2=I$ but $J_1$ and $J_2$ are not necessarily relatively prime. Write $I=P_1^{a_1}\dots P_k^{a_k}$, the factorization of $I$ into prime $\Rbar$-ideals. Then since $J_1\cap J_2=I$, we can write $J_1=P_1^{b_1}\dots P_k^{b_k}$ and $J_2=P_1^{c_1}\dots P_k^{c_k}$, with $b_i,c_i\in\bN_0$ and $a_i=\max\{b_i,c_i\}$ for each $1\leq i\leq k$. Now let $t\in R_1\cap R_2$. Since $R_1\subseteq R+P_i^{b_i}$ and $R_2\subseteq R+P_i^{c_i}$ for each $1\leq i\leq k$, we have that $t\in (R+P_i^{b_i})\cap (R+P_i^{c_i})=R+P_i^{a_i}$ for each $1\leq i\leq k$. Then by the relatively prime case above, $t\in \bigcap_{i=1}^k(R+P_i^{a_i})=R$. Then $R_1\cap R_2\subseteq R$, completing the proof.
    \end{proof}

    One might notice that we have not provided an example of an order which is both ideal-preserving and locally associated but is not associated. Whether such an order exists, i.e. whether the converse of Theorem \ref{ao implies ipo and lao} holds, is currently unknown. However, this converse will hold in some cases, including one which is of particular interest to the results of this paper.

    \begin{theorem}
        \label{radical conductor converse}
        Let $R$ be an order in a number field $K$ with radical conductor ideal $I$. Then $R$ is associated if and only if $R$ is both ideal-preserving and locally associated.
    \end{theorem}
    \begin{proof}
        First, note that Theorem \ref{ao implies ipo and lao} gives us the forward direction of this proof, i.e. that if $R$ is associated, then it is both ideal-preserving and locally associated. We need to show the reverse implication. Assume that $R$ is an order which is both ideal-preserving and locally associated. Let $t\in\Rbar$ and define $J_1:=t\Rbar+I$ and $J_2=IJ_1^{-1}$. Note that since $I$ is assumed to be a radical ideal, $J_1$ and $J_2$ must be relatively prime. By prime avoidance, we can then select some $s\in J_2$ which is relatively prime to $J_1$ (i.e. is not contained in any of the prime ideals dividing $J_1$). Then note that $t+s$ cannot be contained in any prime ideals containing $I$, i.e. $t+s$ is relatively prime to $I$. Since $R$ is locally associated, Condition 3 from Theorem \ref{locally associated orders} tells us that there must exist some $u\in U(\Rbar)$ such that $(t+s)u=r\in R$. Then $tu=r-su^{-1}\in R+J_2$. Since $R$ is ideal-preserving and $t\in J_1$, Theorem \ref{intersect orders} now tells us that $tu\in J_1\cap (R+J_2)\subseteq (R+J_1)\cap (R+J_2)=R+I=R$. Then $R$ is an associated order.
    \end{proof}

    \section{Results for $R$}
    With these tools at hand, we are now ready to consider elasticity in an order in a number field. In particular, we will explore circumstances under which the elasticity of an order might be equal to the elasticity of its integral closure (the ring of algebraic integers). This will allow us to take advantage of what is already known about the elasticity of a number ring. 

    First, we will show that one direction of inequality holds in general. To do so, we need the following lemmas.

    \begin{lemma}
        \label{relatively prime ideal factors}
        \cite{conrad}
        Let $R$ be an order in a number field with conductor ideal $I$. Any $R$-ideal relatively prime to $I$ has unique factorization into prime $R$-ideals relatively prime to $I$. Moreover, all but finitely many prime ideals in $R$ are relatively prime to $I$.
    \end{lemma}

    \begin{lemma}
        \label{classes have primes}
        \cite{picavet}
        Let $R$ be an order in a number field. Every ideal class in $\Cl(R)$ contains infinitely many prime ideals.
    \end{lemma}

    \begin{lemma}
        \label{davenport grows}
        Let $R\subseteq T$ be two orders in the same number field $K$. Then $D(\Cl(R))\geq D(\Cl(T))$, with equality if and only if $\Cl(R)\cong \Cl(T)$.
    \end{lemma}
    \begin{proof}
        Recall from Lemma \ref{surjective class map} that there exists a surjective homomorphism $\tau$ from $\Cl(R)$ onto $\Cl(T)$. Then let $k=D(\Cl(T))$ and $\{g_1,\dots,g_k\}\subseteq \Cl(T)$ be a 0-sequence with no proper 0-subsequence. Since $\tau$ is surjective, there must exist $h_i\in\Cl(R)$ such that $\tau(h_i)=g_i$ for each $1\leq i\leq k$. Since $\{g_1,\dots,g_k\}$ does not have any proper 0-subsequence, neither can the sequence $\{h_1,\dots,h_k\}\subseteq \Cl(R)$. If $\{h_1,\dots,h_k\}$ is a 0-sequence, then $\Cl(R)$ has a 0-sequence of length $k$ with no proper 0-subsequence. If not, then note that (using additive notation), $\{h_1,\dots,h_k,-(h_1+\dots+h_k)\}$ is a 0-sequence of length $k+1$ in $\Cl(R)$. Moreover, it is clear that this 0-sequence cannot have any proper 0-subsequence, as this would force $\{h_1,\dots,h_k\}$ to have a 0-subsequence. Then in this case, $\Cl(R)$ has a 0-sequence of length $k+1$ with no proper 0-subsequence. Since $D(\Cl(R))$ gives the length of the longest 0-sequence in $\Cl(R)$ with no proper 0-subsequence, $D(\Cl(R))\geq k=D(\Cl(T))$.

        Now note that if $\Cl(R)\cong \Cl(T)$, then $D(\Cl(R))=D(\Cl(\Rbar))$ trivially. For the converse, assume that $\Cl(R)$ and $\Cl(T)$ are not isomorphic; that is, that the surjective homomorphism $\tau$ is non-injective. Now construct the sequences $\{g_1,\dots,g_k\}$ and $\{h_1,\dots,h_k\}$ as above. If $\{h_1,\dots,h_k\}$ is not a 0-sequence, we have already shown that we can construct a 0-sequence in $\Cl(R)$ of length $k+1$ with no proper 0-subsequence, and thus $D(\Cl(R))>D(\Cl(T))$. On the other hand, if $\{h_1,\dots,h_k\}$ is a 0-sequence, then let $h_k'\neq h_k$ be another element of $\Cl(R)$ such that $\tau(h_k')=g_k$ (such an element must exist since $\tau$ is non-injective). Then $\{h_1,\dots,h_{k-1},h_k'\}$ is not a 0-sequence, and we again can construct a 0-sequence in $\Cl(R)$ of length $k+1$ with no proper 0-sequence. Thus, if $\Cl(R)$ are non-isomorphic, then $D(\Cl(R))>D(\Cl(T))$.
    \end{proof}

    \begin{theorem}
        \label{elasticity grows}
        Let $R$ be an order in a number field $K$. Then $\rho(R)\geq \rho(\Rbar)$.
    \end{theorem}
    \begin{proof}
        First, note that if $\rho(\Rbar)=1$, then the result is trivial. Otherwise, recall that $\rho(\Rbar)=\frac{D(\Cl(\Rbar))}{2}$.

        Now let $k=D(\Cl(R))>1$ (if $k=1$, then $\abs{\Cl(R)}=\abs{\Cl(\Rbar)}=1$ and thus $\rho(\Rbar)=1$) and $\{[I_1],\dots,[I_k]\}$ be a 0-sequence in $\Cl(R)$ with no 0-subsequence. From the previous lemmas, we can now pick a prime ideal $P_i\in [I_i]$ relatively prime to the conductor ideal $I$ of $R$ for each $1\leq i\leq k$. Furthermore, we can pick a prime ideal $Q_i\in [I_i]^{-1}$ relatively prime to $I$ for each $1\leq i\leq k$. Then let $\alpha\in R$ be a generator of the principal ideal $P_1\dots P_k$; $\beta\in R$ be a generator of the principal ideal $Q_1\dots Q_k$; and for each $1\leq i\leq k$, $\gamma_i\in R$ be a generator of the principal ideal $P_iQ_i$. Also assume without loss of generality that $\alpha\beta=\gamma_1\dots\gamma_k$ (if not, these elements are associates; we can absorb the unit into one of the generators). Since $\alpha$, $\beta$, and each $\gamma_i$ are relatively prime to $I$, then any divisor of these elements must also be relatively prime to $I$.

        Now suppose that $\alpha=ab$ for some $a,b\in R$ with $a\notin U(R)$. Then the principal ideals $aR$ and $bR$ are relatively prime to $I$ and must thus factor uniquely into products of prime $R$-ideals relatively prime to $I$ by Lemma \ref{relatively prime ideal factors}. Since $\alpha R=P_1\dots P_k=aR\cdot bR$, we must have (after possible reordering) that $aR=P_1\dots P_i$ and $bR=P_{i+1}\dots P_k$ for some $1\leq i\leq k$. Then $\{[P_1],\dots,[P_i]\}$ is a 0-subsequence of $\{[I_1],\dots,[I_k]\}$ in $\Cl(R)$. However, since this 0-sequence has no proper 0-subsequence, then $i=k$, meaning that $bR=R$, an empty product of prime ideals. Thus, $b\in U(R)$, so $\alpha$ is irreducible in $R$. By similar arguments, $\beta$ and each $\gamma_i$ must also be irreducible in $R$.

        We now have $\alpha\beta=\gamma_1\dots \gamma_k$, with $\alpha$, $\beta$, and each $\gamma_i$ an irreducible element of $R$. Then by definition of elasticity and using Lemma \ref{davenport grows}, $\rho(R)\geq \frac{k}{2}=\frac{D(\Cl(R))}{2}\geq \frac{D(\Cl(\Rbar))}{2}=\rho(\Rbar)$.
    \end{proof}

    This result tells us that in general, when passing from a ring of algebraic integers to an order contained within, factorization can only get ``worse", not improve. Next, we will explore when equality might hold, i.e. when elasticity will be ``maintained" when moving from the number ring to the order. Our first result along these lines follows almost immediately from Lemma \ref{davenport grows} and the proof of Theorem \ref{elasticity grows}, along with Theorem \ref{rago hfd}, which we will discuss later. This also further demonstrates how the locally associated property discussed earlier is related to factorization.
    
    \begin{corollary}
        Let $R$ be an order in a number field. If $\rho(R)=\rho(\Rbar)$, then $R$ is a locally associated order.
    \end{corollary}

    \begin{proof}
        First, note that if $\abs{\Cl(\Rbar)}=1$, then $\Rbar$ is a UFD (and in particular an HFD). Then $\rho(R)=\rho(\Rbar)$ means that $R$ is an HFD, in which case we may note by Theorem \ref{rago hfd} that $R$ is an associated (and thus locally associated) order.

        If $\abs{\Cl(\Rbar)}> 1$, then using the inequalities shown in the proof of Theorem \ref{elasticity grows}, we have:
        $$\rho(\Rbar)=\frac{D(\Cl(\Rbar))}{2}\leq \frac{D(\Cl(R))}{2}\leq \rho(R).$$
        Thus, if $\rho(R)=\rho(\Rbar)$, then $D(\Cl(\Rbar))=D(\Cl(R))$. Using Lemma \ref{davenport grows}, this tells us that $\Cl(\Rbar)\cong\Cl(R)$, and thus by Condition 6 in Theorem \ref{locally associated orders}, $R$ is a locally associated order.
    \end{proof}
    
    This gives a necessary condition for $\rho(R)=\rho(\Rbar)$. In order to find sufficient conditions, we will need the following lemma describing how units will decompose in more complicated orders.

    \begin{lemma}
        \label{units factor}
        Let $R$ be an associated order in a number field $K$ with conductor ideal $I$. Let $J_1$, $J_2$ be $\Rbar$-ideals containing $I$ such that $J_1$, $J_2$, and $J_3=I(J_1J_2)^{-1}$ are pairwise relatively prime. Denote $R_1=R+J_2J_3$, $R_2=R+J_1J_3$, and $R_3=R+J_3$. Then $U(R_3)=U(R_1)\cdot U(R_2)$, i.e. every unit $u\in U(R_3)$ can be written as a product $u=v_1v_2$, with $v_1\in U(R_1)$ and $v_2\in U(R_2)$.
    \end{lemma}
    \begin{proof}
        First, note that by their definition, $R_1\subseteq R_3$ and $R_2\subseteq R_3$, so $U(R_1)\cdot U(R_2)\subseteq U(R_3)$. We need to show the reverse inclusion. In order to do so, we will first simplify the problem. First, note that if every coset of $\quot{U(R_3)}{U(R)}$ contains an element of $U(R_1)\cdot U(R_2)$, then we are done. This is because then every element $u\in U(R_3)$ can be written as $u=u_1u_2v$, with $u_1\in U(R_1)$, $u_2\in U(R_2)$, and $v\in U(R)\subseteq U(R_2)$. Then $u=u_1(u_2v)\in U(R_1)\cdot U(R_2)$. Thus, it will suffice to show the result modulo $U(R)$, i.e. $$\quot{U(R_3)}{U(R)}=\quot{U(R_1)}{U(R)}\cdot \quot{U(R_2)}{U(R)}.$$
        Now suppose that we have $u_1,v_1\in U(R_1)$ and $u_2,v_2\in U(R_2)$ such that $u_1u_2\equiv v_1v_2\modulo{U(R)}$, i.e. there exists some $w\in U(R)$ such that $u_1u_2=v_1v_2w$. Then $$u_1v_1^{-1}=u_2v_2^{-1}w,$$ with the left-hand side of this identity lying in $U(R_1)$ and the right-hand side lying in $U(R_2)$. Then in fact, both sides lie in $U(R_1)\cap U(R_2)$. Since both $R_1$ and $R_2$ are integral over $R$ and $R_1\cap R_2=R$ by Theorem \ref{intersect orders}, $U(R_1)\cap U(R_2)=U(R)$, and so $u_1\equiv v_1\modulo{U(R)}$ and $u_2\equiv v_2\modulo{U(R)}$. This means that for each distinct choice of cosets $u_1 U(R)\in \quot{U(R_1)}{U(R)}$ and $u_2U(R)\in \quot{U(R_2)}{U(R)}$, we will get a unique product coset $u_1u_2U(R)\in \quot{U(R_3)}{U(R)}$. Then since $\quot{U(R_i)}{U(R)}$ is finite for $i=1,2,3$, it will suffice to show that $$\abs{\quot{U(R_3)}{U(R)}}=\abs{\quot{U(R_1)}{U(R)}}\cdot \abs{\quot{U(R_2)}{U(R)}}.$$

        Now since $R$ is an associated order, we can use Theorems \ref{ao implies ipo and lao} and \ref{inheritance} to conclude that $R$, $R_1$, $R_2$, and $R_3$ are all locally associated orders. In particular, we have from Condition 4 in Theorem \ref{locally associated orders} that $$\abs{\quot{U(\Rbar)}{U(R)}}=\frac{\abs{U(\quot{\Rbar}{I})}}{\abs{U(\quot{R}{I})}},$$ with analogous statements for $R_1$, $R_2$, and $R_3$. Applying these identities, we get the following:
        \begin{align*}
            \abs{\quot{U(R_3)}{U(R)}}&=\frac{\abs{\quot{U(\Rbar)}{U(R)}}}{\abs{\quot{U(\Rbar)}{U(R_3)}}}=\frac{\abs{U(\quot{\Rbar}{I})}\cdot \abs{U(\quot{R_3}{J_3})}}{\abs{U(\quot{\Rbar}{J_3})}\cdot\abs{U(\quot{R}{I})}}\\
            &=\frac{\abs{U(\quot{\Rbar}{J_1})}\cdot\abs{U(\quot{\Rbar}{J_2})}\cdot \abs{U(\quot{\Rbar}{J_3})}\cdot \abs{U(\quot{R_3}{J_3})}}{\abs{U(\quot{\Rbar}{J_3})}\cdot \abs{U(\quot{R}{I})}}\\
            &=\frac{\abs{U(\quot{\Rbar}{J_1})}\cdot\abs{U(\quot{\Rbar}{J_2})}\cdot\abs{U(\quot{R_3}{J_3})}}{\abs{U(\quot{R}{I})}}.
        \end{align*}
        The identities in the first line come from a simple quotient group property and the fact that $R$ and $R_3$ are locally associated orders. The second line follows from applying the Chinese Remainder Theorem to $\quot{\Rbar}{I}$. The final line comes from canceling out the common factor of $\abs{U(\quot{\Rbar}{J_3})}$. By applying largely the same process to $R_1$ and $R_2$, we get:
        $$\abs{\quot{U(R_1)}{U(R)}}=\frac{\abs{U(\quot{\Rbar}{J_1})}\cdot \abs{U(\quot{R_1}{J_2J_3})}}{\abs{U(\quot{R}{I})}},$$
        $$\abs{\quot{U(R_2)}{U(R)}}=\frac{\abs{U(\quot{\Rbar}{J_2})}\cdot \abs{U(\quot{R_1}{J_1J_3})}}{\abs{U(\quot{R}{I})}}.$$
        Finally, we can multiply these last two identities and simplify to obtain:
        $$\abs{\quot{U(R_1)}{U(R)}}\cdot \abs{\quot{U(R_2)}{U(R)}}=\abs{\quot{U(R_3)}{U(R)}}\cdot \frac{\abs{U(\quot{R_1}{J_2J_3})}\cdot \abs{U(\quot{R_2}{J_1J_3})}}{\abs{U(\quot{R_3}{J_3})}\cdot \abs{U(\quot{R}{I})}}$$
        Note that this is precisely the desired identity with an additional fraction multiplied on the right-hand side. To complete the proof, we simply need to show that this fraction is equal to 1. To do so, note that since $R$ is associated, we can again use Theorems \ref{ao implies ipo and lao} and \ref{inheritance} to conclude that $R$ and any order intermediate to $R$ must be ideal-preserving. Then by applying Condition 3 from Theorem \ref{ideal-preserving order}, we have $$\quot{R_1}{J_2J_3}\cong \quot{R+J_2}{J_2}\times \quot{R_3}{J_3},$$
        $$\quot{R_2}{J_1J_2}\cong \quot{R+J_1}{J_1}\times \quot{R_3}{J_3},$$
        $$\quot{R}{I}\cong \quot{R+J_1}{J_1}\times \quot{R+J_2}{J_2}\times \quot{R_3}{J_3}.$$
        Thus,
        $$\quot{R_1}{J_2J_3}\times \quot{R_2}{J_1J_3}\cong \quot{R}{I}\times \quot{R_3}{J_3},$$
        and in particular,
        $$\abs{U(\quot{R_1}{J_2J_3})}\cdot \abs{U(\quot{R_2}{J_1J_3})}=\abs{U(\quot{R}{I})}\cdot \abs{U(\quot{R_3}{J_3})}.$$
        This gives the desired identity above, completing the proof.
    \end{proof}

    With this lemma, we can prove the following result which will lead us to the desired equality of elasticities.

    \begin{theorem}
        \label{irred remain irred}
        Let $R$ be an associated order in a number field $K$ with radical conductor ideal $I$. Then any irreducible element in $R$ remains irreducible in $\Rbar$.
    \end{theorem}
    \begin{proof}
        Let $\alpha\in \Irr(R)$, and write $\alpha=\beta\gamma$ for some $\beta,\gamma\in \Rbar$. To show that $\alpha$ remains irreducible in $\Rbar$, we want to show that either $\beta$ or $\gamma$ must be a unit in $\Rbar$. Since $R$ is an associated order, we can pick some $u,v\in U(\Rbar)$ such that $u\beta,v\gamma\in R$. We will use this notation throughout the proof.

        First, assume that $\alpha$ is relatively prime to $I$, i.e. $\alpha+I\in U(\quot{R}{I})$. Then $uv+I=((u\beta)(v\gamma)+I)(\alpha+I)^{-1}\in \quot{R}{I}$. Then $uv\in U(R)$, so $(uv)\alpha$ is an associate of $\alpha$ in $R$ and thus must also be irreducible in $R$. Then either $u\beta$ or $v\gamma$ must be a unit in $R$, and thus either $\beta$ or $\gamma$ must be a unit in $\Rbar$. Then $\alpha$ remains irreducible in $\Rbar$.

        Now removing the assumption that $\alpha$ is relatively prime to $I$, write $I=J_1J_2J_3$ as follows: let $J_1=\beta\Rbar+I$, the minimal divisor of $I$ which contains $\beta$; let $J_2=\gamma\Rbar+IJ_1^{-1}$, the minimal divisor of $I$ which contains $\gamma$ and is relatively prime to $J_1$; and let $J_3=I(J_1J_2)^{-1}$. Note that as $I$ is radical, $J_1$, $J_2$, and $J_3$ are pairwise relatively prime ideals in $\Rbar$. As in the lemma, we will write $R_1=R+IJ_1^{-1}$, $R_2=R+IJ_2^{-1}$, and $R_3=R+J_3$. Then note that $\alpha$ is relatively prime to $J_3$. By the previous case, this tells us that $uv\in U(R_3)$. We can now use the lemma to conclude that there must exist $w_1\in U(R_1)$ and $w_2\in U(R_2)$ such that $(uv)^{-1}=w_1w_2$. Then $\alpha=(w_1w_2)(uv)\alpha=(w_1u\beta)(w_2v\gamma)$. Note that $w_1u\beta\in R_1\cap J_1\subseteq R$ and $w_2v\gamma\in R_2\cap J_2\subseteq R$. Then $\alpha=(w_1u\beta)(w_2v\gamma)$ is a factorization of $\alpha$ into two elements of $R$, so either $w_1u\beta$ or $w_2v\gamma$ is a unit in $R$. Thus, either $\beta$ or $\gamma$ must be a unit in $\Rbar$, so $\alpha$ remains irreducible in $\Rbar$. 
    \end{proof}

    \begin{theorem}
        \label{elasticity equal}
        Let $R$ be an associated order in a number field $K$ with radical conductor ideal $I$. Then for any nonzero, nonunit $\alpha\in R$, $\rho_R(\alpha)=\rho_{\Rbar}(\alpha)$. Moreover, $\rho(R)=\rho(\Rbar)$.
    \end{theorem}
    \begin{proof}
        By the previous theorem, any irreducible in $R$ will remain irreducible in $\Rbar$. Then for any nonzero, nonunit $\alpha\in R$, any factorization of $\alpha$ into irreducibles in $R$ is also a factorization of $\alpha$ into irreducibles in $\Rbar$. Then $\ell_R(\alpha)\subseteq \ell_{\Rbar}(\alpha)$, so $\rho_R(\alpha)\leq \rho_{\Rbar}(\alpha)$ for every nonzero, nonunit $\alpha\in R$. Note that since $R\subseteq \Rbar$, this also gives $\rho(R)\leq \rho(\Rbar)$. Then by Theorem \ref{elasticity grows}, $\rho(R)=\rho(\Rbar)$.

        Now let $\alpha=\pi_1\dots\pi_k$ be a factorization of $\alpha\in R$ into irreducibles $\pi_i\in \Irr(\Rbar)$. Since $R$ is an associated order, we can find $u_1,\dots,u_k\in U(\Rbar)$ such that $u_i\pi_i\in R$ for $1\leq i\leq k$. Now define ideals $J_1:=\pi_1\Rbar+I$, the minimal divisor of $I$ containing $\pi_i$, and for $2\leq i\leq k$, define $J_i:=\pi_i\Rbar+I(J_1\dots J_{i-1})^{-1}$, the minimal ideal dividing $I$ which contains $\pi_i$ and is relatively prime to each $J_j$ for $1\leq j<i$. Finally, define $J_{k+1}:=I(J_1\dots J_k)^{-1}$ so that $I=J_1\dots J_{k+1}$. Since $I$ is radical, the $J_i$'s are pairwise relatively prime. For ease of notation, we will also define $R_{k+1}:=R+J_{k+1}$, and for $1\leq i\leq k$, $R_i:=R+IJ_i^{-1}$. As seen in the previous proof, since $\alpha$ is relatively prime to $J_{k+1}$, $u_1\dots u_k\in U(R_{k+1})$. Furthermore, we can repeatedly apply Lemma \ref{units factor} to get $$U(R_{k+1})=U(R_1)\dots U(R_k)$$
        Then there exist $v_i\in U(R_i)$ for $1\leq i\leq k$ such that $(u_1\dots u_k)^{-1}=v_1\dots v_k$, and so $\alpha = (v_1u_1\pi_1)\dots(v_ku_k\pi_k)$. Each $v_iu_i\pi_i\in R_i\cap J_i\subseteq R$, so this is a factorization of $\alpha$ in $R$. Furthermore, each $v_iu_i\pi_i$ is an associated of $\pi_i$ in $\Rbar$, ans is thus also irreducible in $\Rbar$ (and must thus be irreducible in $R$). Then any irreducible factorization $\alpha=\pi_1\dots\pi_k$ in $\Rbar$ gives rise in this way to an irreducible factorization $\alpha=(v_1u_1\pi_1)\dots(v_ku_k\pi_k)$ in $R$ of the same length. Then for any nonzero, nonunit $\alpha\in R$, $\ell_{\Rbar}(\alpha)\subseteq\ell_R(\alpha)$, so $\rho_{\Rbar}(\alpha)\leq \rho_R(\alpha)$. Thus, $\rho_{\Rbar}(\alpha)=\rho_R(\alpha)$.
    \end{proof}

    \section{Results for $R[[x]]$}
    With these results shown for orders in a number field, we can now turn our attention to the rings of formal power series over these orders. One may be surprised to see that, with some adjustments to the proofs, we can extend these results to analogous statements for the power series rings. To start, we show once again that we can decompose units in a similar manner to before. In the case of power series rings, it is most convenient to do so in two parts.

    \begin{lemma}
        \label{units factor ps lemma}
        Let $R$ be an associated order in a number field $K$. Let $I$ be the conductor ideal of $R$ and $J_1$, $J_2$ relatively prime $\Rbar$-ideals such that $I=J_1J_2$. Denote $R_1=R+J_1$ and $R_2=R+J_2$. Then $U(\Rbar[[x]])=U(R_1[[x]])\cdot U(R_2[[x]])$.
    \end{lemma}
    \begin{proof}
        First, we recall that given a commutative ring with identity $T$, $f\in T[[x]]$ is a unit if and only if its constant term is a unit in $T$. Thus, an arbitrary element of $U(\Rbar[[x]])$ looks like $u(x)=u_0+b_1x+b_2x^2+\dots$, with $u_0\in U(\Rbar)$ and each $b_i\in \Rbar$. To prove the lemma, we would like to construct $v_1(x)=u_1+c_1x+c_2x^2+\dots\in U(R_1[[x]])$ and $v_2(x)=u_2+d_1x+d_2x^2+\dots\in U(R_2[[x]])$ such that $u=v_1v_2$ (the other inclusion is trivial). By Lemma \ref{units factor}, we immediately know that it is possible to pick constant terms $u_1\in U(R_1)$ and $u_2\in U(R_2)$ such that $u_0=u_1u_2$. Then for each $i\geq 1$, we need to find $c_i\in R_1$ and $d_i\in R_2$ such that the equation $$b_k=c_ku_2+\sum_{i=1}^{k-1}c_id_{k-i}+d_ku_1$$ is satisfied for each $k\geq 1$. Note that this equation can equivalently be stated as $$c_ku_2+d_ku_1=b_k-\sum_{i=1}^{k-1}c_id_{k-i}.$$
        If we assume for some $k\geq 1$ that we have constructed $c_i,d_i$ for $1\leq i<k$ such that each previous equation is satisfied, this becomes a problem of finding $c_k\in R_1$ and $d_k\in R_2$ such that $c_ku_2+d_ku_1$ is equal to some fixed element of $\Rbar$. Then note that $\Rbar\supseteq u_2R_1+u_1R_2\supseteq u_2J_1+u_1J_2= J_1+J_2=\Rbar$. Then finding such $c_k$ and $d_k$ is possible for every $k\geq 1$ (in fact, we could select $c_k\in J_1$ and $d_k\in J_2$ if we so wished), so we can indeed construct $v_1\in U(R_1[[x]])$ and $v_2\in U(R_2[[x]])$ such that $u=v_1v_2$. Thus, $U(\Rbar[[x]])=U(R_1[[x]])\cdot U(R_2[[x]])$.
    \end{proof}

    \begin{lemma}
        \label{units factor ps}
        Let $R$ be an associated order in a number field $K$ with conductor ideal $I$. Let $J_1$, $J_2$ be $\Rbar$-ideals containing $I$ such that $J_1$, $J_2$, and $J_3=I(J_1J_2)^{-1}$ are pairwise relatively prime. Denote $R_1=R+J_2J_3$, $R_2=R+J_1J_3$, and $R_3=R+J_3$. Then $U(R_3[[x]])=U(R_1[[x]])\cdot U(R_2[[x]])$.
    \end{lemma}
    \begin{proof}
        Similarly to the proof of Lemma \ref{units factor}, it is trivial to show that $U(R_1[[x]])\cdot U(R_2[[x]])\subseteq U(R_3[[x]])$. We need to show the reverse inclusion. Let $u\in U(R_3[[x]])$. Since this is an element in $U(\Rbar[[x]])$, we can use the previous lemma to construct $v_1\in U((R+J_1)[[x]])$ and $v_2\in U((R+J_2)[[x]])$ such that $u=v_1v_2$ (the lemma would actually give $v_2\in U((R+J_2J_3)[[x]])$, but we do not need to be this specific). Now note that since $J_1[[x]]$, $J_2[[x]]$, and $J_3[[x]]$ are pairwise relatively prime ideals of $\Rbar[[x]]$, the Chinese Remainder Theorem gives us the following:
        $$\quot{\Rbar}{I}[[x]]\cong \quot{\Rbar}{J_1}[[x]]\times \quot{\Rbar}{J_2}[[x]]\times \quot{\Rbar}{J_3}[[x]].$$
        Then we will pick some coset representative $w\in \Rbar[[x]]$ of the congruence class modulo $I[[x]]$ such that $w\equiv 1\modulo{J_1[[x]]}$, $w\equiv 1\modulo{J_2[[x]]}$, and $w\equiv v_1^{-1}\modulo{J_3[[x]]}$. Note that although $w$ itself is not necessarily a unit in $\Rbar[[x]]$, its constant term, say $w_0$, is congruent to a unit in $U(\Rbar)$ modulo $J_1$, $J_2$, and $J_3$, so $w_0+I\in U(\quot{\Rbar}{I})$. Since $R$ is an associated (and thus locally associated) order, there exists $r+I\in U(\quot{R}{I})$ and $\beta\in I$ such that $w_0r+\beta\in U(\Rbar)$. Define $w'=wr+\beta\in U(\Rbar[[x]])$ and $s\in R$ such that $s+I=(r+I)^{-1}\in U(\quot{R}{I})$.

        Now note that $u=v_1v_2=(v_1w')(v_2w'^{-1})$. Moreover:
        $$v_1w'\equiv v_1(wr+\beta)\equiv v_1r\modulo{J_1[[x]]}\implies v_1w'\in (R+J_1)[[x]];$$
        $$v_1w'\equiv v_1(wr+\beta)\equiv v_1v_1^{-1}r\equiv r\modulo{J_3[[x]]}\implies v_1w'\in R_3[[x]];$$
        $$v_2w'^{-1}\equiv v_2(wr+\beta)^{-1}\equiv v_2s\modulo{J_2[[x]]}\implies v_2w'^{-1}\in (R+J_2)[[x]]$$
        $$v_2w'^{-1}\equiv v_2(wr+\beta)^{-1}\equiv v_2v_1s\equiv us\modulo{J_3[[x]]}\implies v_2 w'^{-1}\in R_3[[x]].$$
        Then since $R$ is an associated (and thus ideal-preserving) order and $J_1$, $J_2$, and $J_3$ are pairwise relatively prime, $v_1w'\in (R+J_1)[[x]]\cap R_3[[x]]=R_2[[x]]$ and $v_2w'^{-1}\in (R+J_2)[[x]]\cap R_3[[x]]=R_1[[x]]$. Furthermore, since $v_1$, $v_2$, and $w'$ are all units in $\Rbar[[x]]$, then $v_1w'\in U(R_2[[x]])$ and $v_2w'\in U(R_1[[x]])$, with $u=(v_1w')(v_2w'^{-1})$. Thus, $u\in U(R_1[[x]])\cdot U(R_2[[x]])$, meaning that $U(\Rbar[[x]])=U(R_1[[x]])\cdot U(R_2[[x]])$.
    \end{proof}

    The final lemma we will need to show the desired elasticity result for power series is one that relates back to the concept of an associated subring. This will allow us to take advantage of a very nice relationship between the rings $R[[x]]$ and $\Rbar[[x]]$.

    \begin{lemma}
        \label{associated ps}
        Let $R$ be an associated order in a number field $K$ with radical conductor ideal $I$. Then $R[[x]]$ is an associated subring of $\Rbar[[x]]$.
    \end{lemma}
    \begin{proof}
        Let $a(x)=a_0+a_1x+a_2x^2+\dots\in \Rbar[[x]]$. We want to show that there exist $r(x)=r_0+r_1x+r_2x^2+\dots\in R[[x]]$ and $u_0+b_1x+b_2x^2+\dots\in U(\Rbar[[x]])$ such that $a=ru$. First, note that since $R$ is an associated order, we can find constant terms $r_0\in R$ and $u_0\in U(\Rbar)$ such that $a_0=r_0u_0$.

        Working toward an inductive argument, we will start by assuming that $I$ is prime. If $a\in I[[x]]$, then we could choose $r=a$ and $u=1$. Otherwise, $a\notin I[[x]]$; for now, we will also assume that $a_0\in I$. Since $a_0=r_0u_0$, this means that $r_0\notin I$; since $I$ is prime, $r_0+I\in U(\quot{\Rbar}{I})$. We now need to construct $r_i\in R$ and $b_i\in \Rbar$ for $i\in \bN$ such that for each $k\in\bN$,
        $$a_k=r_0b_k+\sum_{i=1}^{k-1}r_ib_{k-i}+r_ku_0;$$ equivalently,
        $$r_0b_k+r_ku_0=a_k-\sum_{i=1}^{k-1}r_ib_{k-i}.$$
        Assume that for some $k\in\bN$, we have selected $r_i$ and $b_i$ for every $1\leq i<k$ such that all previous equations are satisfied. Then note that $\Rbar\supseteq r_0\Rbar+u_0R\supseteq r_0\Rbar+I=\Rbar$. Then selecting $r_k\in R$ and $b_k\in\Rbar$ to satisfy the $k^{th}$ equation is always possible, so we can construct $r\in R[[x]]$ and $u\in U(\Rbar[[x]])$ such that $a=ru$.

        Now assume that $I$ is prime and $a\notin I[[x]]$, but $a_0\in I$. Since $a\notin I[[x]]$, there is some minimal $j\in\bN$ such that $a_j\notin I$ (i.e. $a_i\in I$ for every $0\leq i<j$). Then write $a(x)=(a_0+a_1x+\dots+a_{j-1}x^{j-1})+x^j(a_j+a_{j+1}x+a_{j+2}x^2+\dots)=b+cx^j$, with $b\in I[x]$ and $c\in \Rbar[[x]]$ having constant term $a_j\notin I$. By the previous case, there must exist $r\in R[[x]]$ and $u\in U(\Rbar[[x]])$ such that $c=ru$. Then $a=b+cx^j=(bu^{-1}+rx^j)u$, with $bu^{-1}+rx^j\in R[[x]]$ and $u\in U(\Rbar[[x]])$. This completes the case when $I$ is prime.

        We can now tackle the case when $I$ is not prime (but is still radical). For the inductive argument, assume that for every $\Rbar$-ideal $J$ properly dividing $I$, $(R+J)[[x]]$ is an associated subring of $\Rbar[[x]]$ (note that any such $R+J$ is an associated order with radical conductor ideal $J$). Since $I$ is not prime, we can write $I=J_1J_2$, with neither $J_1$ nor $J_2$ equal to $I$; since $I$ is radical, $J_1$ and $J_2$ are relatively prime. For ease of notation, denote $R_1=R+J_1$ and $R_2=R+J_2$. By the inductive hypothesis, we can select $r_1\in R_1[[x]]$ and $u\in U(\Rbar[[x]])$ such that $a=r_1u$. Furthermore, we can select $r_2\in R_2[[x]]$ and $v\in U(\Rbar[[x]])$ such that $r_1=r_2v$. Finally, by Lemma \ref{units factor ps lemma}, we can select $w_1\in U(R_1[[x]])$ and $w_2\in R_2[[x]])$ such that $v=w_1w_2$. Then note that $r_1=r_2v=r_2w_1w_2\implies r_1w_1^{-1}=r_2w_2$. Since the left-hand side of this equality lies in $R_1[[x]]$ and the right-hand side lies in $R_2[[x]]$, both must actually lie in $R_1[[x]]\cap R_2[[x]]=R[[[x]]$. Then $a=r_1u=r_2vu=(r_2w_2)(w_1u)\in R[[x]]\cdot U(\Rbar[[x]])$. Thus, $R[[x]]$ is an associated subring of $\Rbar[[x]]$.
    \end{proof}

    Now, as in the earlier case of the orders themselves, we can consider the irreducible elements in $R[[x]]$. Then, we can use this result to tackle the question of elasticity.

    \begin{theorem}
        \label{irred remain irred ps}
        Let $R$ be an associated order in a number field $K$ with radical conductor ideal $I$. Then any irreducible element in $R[[x]]$ remains irreducible in $\Rbar[[x]]$.
    \end{theorem}
    \begin{proof}
        Let $f$ be an irreducible power series in $R[[x]]$, and write $f=gh$ for $g,h\in\overline{R}[[x]]$. To show that $f$ remains irreducible in $\overline{R}[[x]]$, we want to show that either $g$ or $h$ is a unit. The previous theorem tells us that $R[[x]]$ is an associated subring of $\overline{R}[[x]]$, so there exist $u,v\in U(\overline{R}[[x]])$ such that $ug,vh\in R[[x]]$. We will use this notation throughout the rest of the proof.
        
        First, we will consider the case when $f\notin I[[x]]$ and $I$ is a prime $\overline{R}$-ideal. We write $uv=u_0+b_1x+b_2x^2+\dots\in U(\overline{R}[[x]])$, $f=r_0+r_1x+r_2x^2+\dots\in R[[x]]$, and $(uv)f=a_0+a_1x+a_2x^2+\dots\in R[[x]]$ (note that $(uv)f$ must lie in $R[[x]]$ as the product of $ug$ and $vh$, both of which lie in $R[[x]]$). Since $f\notin I[[x]]$, there is some minimal $n\in\bN_0$ such that $r_n\notin I$ (i.e. $r_i\in I$ for $0\leq i<n$). Then
        $$a_n=u_0r_n+\sum_{i=1}^nb_ir_{n-i}\implies u_0r_n=a_n-\sum_{i=1}^nb_ir_{n-i}.$$ Note that $a_n\in R$ and $r_{n-i}\in I$ for $1\leq i\leq n$, so the right-hand side of this equation lies in $R$. Furthermore, since $I$ is prime and $r_n\notin I$, $r_n+I\in U(\large\sfrac{R}{I})$, so $$u_0+I=\left(a_n-\sum_{i=1}^nb_ir_{n-i}+I\right)(r_n+I)^{-1}\in \large\sfrac{R}{I}.$$ Then $u_0\in U(R)$.
        
        Now for some $k>n$, we will assume that $u_0\in U(R)$ and $b_i\in R$ for $1\leq i<k-n$. Similarly to before, we write
        $$a_k=u_0r_k+\sum_{i=1}^{k-n-1}b_ir_{k-i}+b_{k-n}r_n+\sum_{i=k-n+1}^kb_ir_{k-i};$$
        rearranging, this gives $$b_{k-n}r_n=a_k-u_0r_k-\sum_{i=1}^{k-n-1}b_ir_{k-i}-\sum_{i=k-n+1}^kb_ir_{k-i}.$$
        On the right-hand side of this equation, note that $a_k\in R$, $u_0r_k\in R$, $b_ir_{k-i}\in R$ for $1\leq i\leq k-n-1$, and $r_{k-i}\in I$ for $i\geq k-n+1$. Then the right-hand side of this equation is in $R$, so we can proceed as before to multiply by the inverse of $r_n$ modulo $I$ to get $$b_{k-n}+I=\left(a_k-u_0r_k-\sum_{i=1}^{k-n-1}b_ir_{k-i}-\sum_{i=k-n+1}^kb_ir_{k-i}+I\right)(r_n+I)^{-1}\in \large\sfrac{R}{I}.$$ Then $b_i\in R$ for every $i\in \bN$, so $uv\in U(R[[x]])$. Then $uvf$ is an associate of $f$ in $R[[x]]$, and must thus remain irreducible in $R[[x]]$. Therefore, either $ug$ or $vh$ must be a unit in $R[[x]]$, meaning that either $g$ or $h$ must be a unit in $\overline{R}[[x]]$. Thus, $f$ remains irreducible in $\overline{R}[[x]]$.
        
        Now removing the assumption that $I$ is prime, we will split $I$ into relatively prime factors much as in the proof of Theorem \ref{irred remain irred}: let $J_1$ be the minimal divisor of $I$ such that $J_1[[x]]$ contains $g$; let $J_2$ be the minimal divisor of $IJ_1^{-1}$ such that $J_2[[x]]$ contains $h$; and let $J_3=I(J_1J_2)^{-1}$. As before, we will denote $R_1=R+IJ_1^{-1}$, $R_2=R+IJ_2^{-1}$, and $R_3=R+J_3$. Since $P[[x]]$ is a prime $\overline{R}[[x]]$-ideal for any prime ideal $P$, note that $\alpha$ is not contained in $P[[x]]$ for any prime divisor $P$ of $J_3$. By the previous case, this tells us that $uv\in (R+P)[[x]]$ for every prime divisor $P$ of $J_3$, and thus $uv\in \bigcap_{P|J_3}(R+P)[[x]]=R_3[[x]]$. Then since $uv\in U(R_3[[x]])$, we can use Theorem \ref{units factor ps} to conclude that there must exist $w_1\in U(R_1[[x]])$ and $w_2\in U(R_2[[x]])$ such that $(uv)^{-1}=w_1w_2$. Then $f=(uv)^{-1}(ug)(vh)=(w_1ug)(w_2vh)$, with $w_1ug\in R_1[[x]]\cap J_1[[x]]\subseteq R[[x]]$ and $w_2vh\in R_2[[x]]\cap J_2[[x]]\subseteq R[[x]]$. Then $f=(w_1ug)(w_2vh)$ is a factorization of $f$ in $R[[x]]$, so either $w_1ug$ or $w_2vh$ is a unit in $R[[x]]$. Then either $g$ or $h$ must be a unit in $\overline{R}[[x]]$, meaning that $f$ is irreducible in $\overline{R}[[x]]$.
    \end{proof}

    \begin{theorem}
        \label{elasticity equal ps}
        Let $R$ be an associated order in a number field $K$ with radical conductor ideal $I$. Then for any nonzero, nonunit $f\in R[[x]]$, $\rho_{R[[x]]}(f)=\rho_{\Rbar[[x]]}(f)$. Moreover, $\rho(R[[x]])=\rho(\Rbar[[x]])$.
    \end{theorem}
    \begin{proof}
        By the previous theorem, any irreducible in $R[[x]]$ remains irreducible in $\overline{R}[[x]]$. Then for any nonzero, nonunit $f\in R[[x]]$, any factorization of $f$ into $R[[x]]$-irreducibles is also a factorization of $f$ into $\overline{R}[[x]]$-irreducibles. Then $\ell_{R[[x]]}(f)\subseteq\ell_{\overline{R}[[x]]}(f)$, so $\rho_{R[[x]]}(f)\leq\rho_{\overline{R}[[x]]}(f)$ for every nonzero, nonunit $f\in R[[x]]$. Note that this also gives $\rho(R[[x]])\leq \rho(\overline{R}[[x]])$.
        
        Now let $f=g_1\dots g_k$ be a factorization of some nonzero, nonunit $f\in R[[x]]$ into irreducibles $g_i\in \Irr(\overline{R}[[x]])$. By Lemma \ref{associated ps}, we can find $u_1,\dots,u_k\in U(\overline{R}[[x]])$ such that $u_ig_i\in R[[x]]$ for every $1\leq i\leq k$. Now let $J_1$ be the minimal divisor of $I$ such that $J_1[[x]]$ contains $g_1$, and for $2\leq i\leq k$, let $J_i$ be the minimal divisor of $I(J_1\dots J_{i-1})^{-1}$ such that $J_i[[x]]$ contains $g_i$. Finally, let $J_{k+1}=I(J_1\dots J_k)^{-1}$ so that $J_1,\dots, J_{k+1}$ are pairwise relatively prime ideals such that $I=J_1\dots J_{k+1}$ and $g_i\in J_i[[x]]$ for $1\leq i\leq k$. For ease of notation, we will denote $R_{k+1}=R+J_{k+1}$ and $R_i=R+IJ_i^{-1}$ for $1\leq i\leq k$. Then as in the proof of the previous theorem, since $f$ is not contained in $P[[x]]$ for any prime $\overline{R}$-ideal $P$ dividing $J_{k+1}$, $u_1\dots u_k\in U(R_{k+1}[[x]])$.
        
        By repeatedly applying Theorem \ref{units factor ps}, we have that $$U(R_{k+1}[[x]])=U(R_1[[x]])\dots U(R_k[[x]]).$$
        Then there exist $v_i\in U(R_i[[x]])$ for $1\leq i\leq k$ such that $(u_1\dots u_k)^{-1}=v_1\dots v_k$. Thus, $f=(v_1u_1g_1)\dots (v_ku_kg_k)$. Note that each $v_iu_ig_i\in R_i[[x]]\cap J_i[[x]]\subseteq R[[x]]$; furthermore, since each $v_iu_ig_i$ is an associate of $g_i\in \Irr(\overline{R}[[x]])$, each of these elements is irreducible in $\overline{R}[[x]]$ (and thus also irreducible in $R[[x]]$). Then the irreducible factorization $f=g_1\dots g_k$ in $\overline{R}[[x]]$ gave rise to an irreducible factorization $f=(v_1u_1g_1)\dots(v_ku_kg_k)$ in $R[[x]]$ of the same length. Then $\ell_{\overline{R}[[x]]}(f)\subseteq \ell_{R[[x]]}(f)$, so $\rho_{\overline{R}[[x]]}(f)\leq \rho_{R[[x]]}(f)$ for every nonzero, nonunit $f\in R[[x]]$. Thus, $\rho_{\overline{R}[[x]]}(f)=\rho_{R[[x]]}(f)$.
        
        All that remains to show is that $\rho(R[[x]])=\rho(\overline{R}[[x]])$. We already know that $\rho(R[[x]])\leq \rho(\overline{R}[[x]])$; furthermore, for any nonzero, nonunit $f\in R[[x]]$, we know that $\rho_{\overline{R}[[x]]}(f)=\rho_{R[[x]]}(f)$. Then if we can show that for any nonzero, nonunit $f\in \overline{R}[[x]]$, there exists $g\in R[[x]]$ such that $\rho_{\overline{R}[[x]]}(f)=\rho_{R[[x]]}(g)$, this would complete the proof. To do so, recall that by Lemma \ref{associated ps}, there exists $u\in U(\overline{R}[[x]])$ such that $uf\in R[[x]]$. Then since $f$ and $uf$ are associates in $\overline{R}[[x]]$, we know that $\rho_{\overline{R}[[x]]}(uf)=\rho_{\overline{R}[[x]]}(f)$. Furthermore, the previous argument tells us that since $uf\in R[[x]]$, $\rho_{R[[x]]}(uf)=\rho_{\overline{R}[[x]]}(uf)$. Then for any nonzero, nonunit $f\in \overline{R}[[x]]$, there is some $g\in R[[x]]$ (in this construction, $g=uf$) such that $\rho_{\overline{R}[[x]]}(f)=\rho_{R[[x]]}(g)$. Then $\rho(\overline{R}[[x]])\leq\rho(R[[x]])$, so $\rho(\overline{R}[[x]])=\rho(R[[x]])$.
    \end{proof}

    Using Theorems \ref{elasticity equal} and \ref{elasticity equal ps}, we can compare the elasticity of an order $R$ to that of its integral closure as well as the elasticity of $R[[x]]$ to that of $\Rbar[[x]]$. In cases when a relationship is known to exist between $\rho(\Rbar)$ and $\rho(\Rbar[[x]])$, we will have a relationship among all four of these elasticities. In particular, as we will see in the next section, we will be able to conclude that if $R$ is an associated order with radical conductor ideal, then $R$ will be an HFD if and only if $R[[x]]$ is an HFD.

    \section{Half-Factorial Orders}
    One will recall that a half-factorial domain is characterized by being an atomic domain with elasticity 1. Then as one might expect, the previous results will apply nicely to the case when the order $R$ in question is an HFD. Before seeing these applications, it will help to first consider the recent characterization of half-factorial orders in an algebraic number field from \cite{rago}. For convenience, we present this result using the notation and terminology found in this paper rather than the original.

    \begin{theorem}
        \label{rago hfd}
        \cite{rago} Let $K$ be a number field and $R$ an order in $K$ with conductor ideal $I$. Let $I=P_1^{a_1}\dots P_k^{a_k}$ be the factorization of $I$ into prime $\Rbar$-ideals, and denote $Q_i=R\cap P_i$ for each $1\leq i\leq k$. Then $R$ is an HFD if and only if the following properties hold:
        \begin{enumerate}
            \item $\Rbar$ is an HFD;
            \item $R$ is an associated order;
            \item For each $1\leq i\leq k$, $a_i\leq 4$, and letting $\pi_i$ be an arbitrary prime element of $\Rbar_{Q_i}$, $v_{\pi_i}(\Irr(R_{Q_i}))\subseteq \{1,2\}$. If $P_i$ is a principal ideal, then $a_i\leq 2$ and $v_{\pi_i}(\Irr(R_{Q_i})=\{1\}$.
        \end{enumerate}
        Here, $R_{Q_i}$ refers to the localization of $R$ at the prime ideal $Q_i$ (with $\Rbar_{Q_i}$ analogously defined) and $v_{\pi_i}$ is the valuation associated with the element $\pi_i$.
    \end{theorem}

    This result gives a full characterization of which orders in an algebraic number field are HFDs. However, it may at times be inconvenient to use. For our purposes in this paper, we will instead use the following characterization which follows from the above characterization and a lemma also found in \cite{rago}.

    \begin{theorem}
        \label{hfd iff irreducibles}
        Let $R$ be an order in a number field $K$. Then $R$ is an HFD if and only if the following properties hold.
        \begin{enumerate}
            \item $\Rbar$ is an HFD;
            \item $R$ is an associated order;
            \item $\Irr(R)\subseteq \Irr(\Rbar)$; that is, any element which is irreducible in $R$ remains irreducible when considered as an element of $\Rbar$.
        \end{enumerate}
    \end{theorem}

    Notably, these results tell us that any half-factorial order in a number field must automatically be an associated order. Then in order to apply the elasticity results in Theorems \ref{elasticity equal} and \ref{elasticity equal ps}, we only need to show that the conductor ideal is radical.

    Next, we show a result that will help us greatly when considering the elasticity of the ring of formal power series over an order in a number field. One will note that our previous results relate the elasticity of $R$ to that of $\Rbar$ and the elasticity of $R[[x]]$ to that of $\Rbar[[x]]$; this will allow us to bridge the gap between these pairs of rings. Before proving the theorem of interest, we must first reference a series of lemmas. In the following results, we will use $\Cl(R)$ to denote the ideal class group of a ring $R$ and $\DivCl(R)$ to denote the divisor class group of a ring $R$. For precise definitions and discussion of these two groups, see \cite{fossum}.

    \begin{lemma}
        \label{noetherian krull regular ps}
        \cite{matsumura}
        Let $R$ be a ring. Then:
        \begin{enumerate}
            \item If $R$ is a Noetherian ring, so is $R[[x]]$.
            \item If $R$ is a Krull domain, so is $R[[x]]$.
            \item If $R$ is a regular ring, so is $R[[x]]$.
        \end{enumerate}
    \end{lemma}

    \begin{lemma}
        \label{class numbers equal}
        \cite{gortz}
        Let $R$ be a locally factorial ring. Then $\Cl(R)\cong \DivCl(R)$. In particular, this holds when $R$ is a regular ring.
    \end{lemma}

    \begin{lemma}
        \label{divcl ps bijection}
        \cite{claborn}
        Let $R$ be a regular Noetherian domain. Then the natural mapping $\DivCl(R)\to \DivCl(R[[x]])$ is a bijection.
    \end{lemma}

    \begin{lemma}
        \label{krull ufd iff}
        \cite{fossum}
        Let $R$ be an integral domain. Then $R$ is a Krull domain with $\abs{\DivCl(R)}=1$ if and only if $R$ is a UFD.
    \end{lemma}

    \begin{lemma}
        \label{class number 2}
        \cite{zaks1}
        Let $R$ be a Krull domain. If $\abs{\Cl(R)}=2$, then $R$ is an HFD. 
    \end{lemma}

    \begin{theorem}
        \label{integrally closed hfd iff ps}
        Let $K$ be a number field and $R$ its ring of algebraic integers, i.e. $R$ is the maximal order in the number field $K$. The following are equivalent:
        \begin{enumerate}
            \item $R$ is an HFD.
            \item $R[[x_1,\dots,x_k]]$ is an HFD for some $k\in\bN$.
            \item $R[[x_1,\dots,x_k]]$ is an HFD for every $k\in\bN$.
        \end{enumerate}
    \end{theorem}
    \begin{proof}
        First, note that $3\implies 2$ immediately. To complete the proof, we will show $2\implies 1$, then $1\implies 3$.

        Assume that Condition 2 holds, i.e. there is some $k\in\bN$ such that $R[[x_1,\dots,x_k]]$ is an HFD. Since $R$ is Noetherian, it must also be atomic. Then let $\alpha$ be a nonzero, nonunit element of $R$ and consider two irreducible factorizations of $\alpha$ in $R$, $\alpha=\sigma_1\dots\sigma_m=\tau_1\dots\tau_n$. We now note that each $\sigma_i$ and $\tau_j$ must remain irreducible in $R[[x_1,\dots,x_k]]$; this is easy to see by considering the constant terms of potential factorizations in $R[[x_1,\dots,x_k]]$. Then since $R[[x]]$ is an HFD, we have that $m=n$. Thus, $R$ is an HFD, so $2\implies 1$. Note that this portion of the proof will hold more generally, as it only depended on $R$ being an atomic domain.

        Finally, assume that $R$ is an HFD; we want to show that for any $k\in\bN$, $R[[x_1,\dots,x_k]]$ must be an HFD. First, note that since $R$ is a Dedekind domain, it is Noetherian, Krull, and regular; then by Lemma \ref{noetherian krull regular ps}, so is $R[[x]]$, and by induction, $R[[x_1,\dots,x_k]]$. Furthermore, since $R$ is an HFD number ring, $\Cl(R)\leq 2$. Now since $R$ and $R[[x_1,\dots,x_k]]$ are regular rings, Lemma \ref{class numbers equal} tells us that $\Cl(R)\cong \DivCl(R)$ (i.e. $\abs{\DivCl(R)}\leq2$) and $\Cl(R[[x_1,\dots,x_k]])\cong\DivCl(R[[x_1,\dots,x_k]])$. Again using the fact that $R$ is a regular Noetherian ring, we can use Lemma \ref{divcl ps bijection} to conclude that $\abs{\DivCl(R[[x]])}=\abs{\DivCl(R)}\leq 2$. By induction, $\abs{\DivCl(R[[x_1,\dots,x_k]])}\leq 2$. Finally, since $R[[x_1,\dots,x_k]]$ is a Krull domain with $\abs{\Cl(R[[x_1,\dots,x_k]])}=\abs{\DivCl(R[[x_1,\dots,x_k]])}\leq 2$, Lemmas \ref{krull ufd iff} and \ref{class number 2} tell us that $R[[x_1,\dots,x_k]]$ must be an HFD. Thus, $1\implies 3$.
    \end{proof}

    Using this result, we get the following pertaining to non-integrally closed orders.

    \begin{theorem}
        \label{radical hfd ps}
        Let $R$ be an order in a number field $K$ with radical conductor ideal $I$. Then $R$ is an HFD if and only if $R[[x]]$ is an HFD.
    \end{theorem}
    \begin{proof}
        First, note that if $R[[x]]$ is an HFD, then $R$ will be an HFD as well; this part of the proof follows exactly as in the proof of Theorem \ref{integrally closed hfd iff ps}. Then assume that $R$ is an HFD, i.e. $\rho(R)=1$. Note by Theorem \ref{hfd iff irreducibles} that since $R$ is an HFD, it must also be an associated order. Then Theorem \ref{elasticity equal} tells us that since $R$ is an associated order with radical conductor ideal, $\rho(\Rbar)=\rho(R)=1$, and thus $\Rbar$ is an HFD. Now Theorem \ref{integrally closed hfd iff ps} tells us that $\Rbar[[x]]$ must be an HFD, and Theorem \ref{elasticity equal ps} tells us that $\rho(R[[x]])=\rho(\Rbar[[x]])=1$. Then $R[[x]]$ is an HFD.
    \end{proof}

    One will recall from Theorem \cite{halter-koch} that in a quadratic number field, any half-factorial order has radical conductor ideal. Then we get the following corollary immediately.

    \begin{corollary}
        Let $R$ be an order in a quadratic number field $K$. Then $R$ is an HFD if and only if $R[[x]]$ is an HFD.
    \end{corollary}

    \section{Non-Radical Conductor Ideals}
    Throughout the major results of this paper, one will note that we rely heavily on requiring the conductor ideal $I$ of the order $R$ in question to be radical. A natural question to ask, then, is whether this is a necessary condition. In particular, we will focus on the case when $R$ is an HFD. The two questions we will answer here are the following:
    \begin{enumerate}
        \item Let $R$ be a half-factorial order in a number field $K$ with conductor ideal $I$. Does $I$ necessarily have to be radical?
        \item Let $R$ be a half-factorial order in a number field $K$. Does $R[[x]]$ necessarily have to be an HFD?
    \end{enumerate}
    
    In Rago's characterization of half-factorial orders in Theorem \ref{rago hfd}, it would seem as though a half-factorial order with non-radical conductor ideal should exist, though it is not immediately obvious that such an example must exist. In a revision of his original paper, Rago provides the following example.

    \begin{example}
        \cite{rago}
        Let $K=\bQ[\alpha]$, with $\alpha$ a root of $x^3-8x-19$. Letting $P=(2,1+\alpha+\alpha^2)$ (one of the non-principal primes lying over 2), $I=P^2$, and $R=\bZ+I$, we have that $R$ is a half-factorial order in $K$ with non-radical conductor ideal $I$.
    \end{example}

    Here, we will provide a similar example using different methods than those in Rago's paper. To do so, we will first show the following result.

    \begin{theorem}
        \label{non-radical hfd}
        Let $R$ be an associated order in a number field $K$ such that $\abs{\Cl(\Rbar)}=2$, i.e. $\Rbar$ is an HFD which is not a UFD. Also assume that $I=P^2$ is the conductor ideal of $R$ for some non-principal $\Rbar$-ideal $P$. Then $R$ is an HFD.
    \end{theorem}
    \begin{proof}
        Since $\Rbar$ is an HFD and $R$ is an associated order, Theorem \ref{hfd iff irreducibles} tells us that it will suffice to show that every irreducible in $R$ remains irreducible in $\Rbar$. Let $\alpha\in \Irr(R)$, and for now, assume that $\alpha\notin I$. Suppose that we can write $\alpha=\beta\gamma$ for some $\beta,\gamma\in R$. Since $\alpha\notin I=P^2$, then at least one of $\beta$ or $\gamma$ must lie outside of $P$; without loss of generality, assume $\gamma\notin P$. We will also define $U_\alpha:=\{u\in U(\Rbar)|u\alpha\in R\}$, with $U_\beta$ and $U_\gamma$ defined similarly. Note that since $R$ is an associated order, none of these three sets are empty.

        Let $v\in U_\gamma$, i.e. $v\gamma\in R$. Then if $u\in U_\beta$, note that $(uv)\alpha=(u\beta)(v\gamma)\in R$, so $uv\in U_\alpha$. Thus, $vU_\beta\subseteq U_\alpha$. Similarly, if $u\in U_\alpha$, then $u\alpha=u\beta\gamma=(uv^{-1}\beta)(v\gamma)\in R$. Since $v\gamma\in R\backslash P$, we know that $v\gamma+I\in U(\quot{R}{I})$. Then $uv^{-1}\beta+I=(u\alpha+I)(v\gamma+I)^{-1}\in \quot{R}{I}$, so $uv^{-1}\beta\in R$. Then $v^{-1}u\in U_\beta$, so $v^{-1}U_\alpha\subseteq U_\beta\implies U_\alpha\subseteq vU_\beta$. Therefore, $U_\alpha=vU_\beta$ for any $v\in U_\gamma$.

        Now note that since $\alpha\in R$, $1\in U_\alpha$. Then $v^{-1}\in U_\beta$, so $\alpha=(v^{-1}\beta)(v\gamma)$, with both $v^{-1}\beta$ and $v\gamma$ lying in $R$. Since $\alpha$ is irreducible in $R$, this means that either $v^{-1}\beta$ or $v\gamma$ must be a unit in $R$, and thus either $\beta$ or $\gamma$ must be a unit in $\Rbar$. Then $\alpha$ remains irreducible in $\Rbar$.

        All that remains to show is that any irreducible $\alpha\in \Irr(R)$ which lies in $I$ remains irreducible in $\Rbar$. To do so, note that since $I=P^2$, with $P$ a non-principal prime in the HFD $\Rbar$, $I$ must be a principal ideal generated by some $\pi\in \Irr(\Rbar)$. Then since $\alpha\in I=(\pi)$, there must be some $\beta\in \Rbar$ such that $\alpha=\beta\pi$. Letting $u\in U(\Rbar)$ such that $u\beta\in R$, we have $\alpha=(u\beta)(u^{-1}\pi)$, with $u\beta\in R$ and $u^{-1}\pi\in I\subseteq R$. Since $\alpha$ is irreducible in $R$ and $u^{-1}\pi$ is not a unit, it follows that $u\beta$ must be a unit, and thus $\alpha$ is an associate of the irreducible element $\pi$ in $\Rbar$. Then $\alpha$ remains irreducible in $\Rbar$, so $R$ is an HFD.
    \end{proof}

    Using this theorem, we can now present the following example of a half-factorial order in a number field with non-radical conductor ideal.

    \begin{example}
        \label{non-radical hfd example}
        Let $K=\bQ[\alpha]$, with $\alpha$ a root of $x^3+4x-1$. From the database at \cite{lmfdb}, we get the ring of algebraic integers in $K$ is $\cO_K=\bZ[\alpha]$, $\abs{\Cl(K)}=2$, and $\cO_K$ admits the fundamental unit $\alpha$. In $\cO_K$, the rational prime $3$ factors as $3\Rbar=(3,1+\alpha)(3,2+2\alpha+\alpha^2)$, with both of these prime factors being non-principal. Then let $I=(3,2+2\alpha+\alpha^2)^2=(2-4\alpha+\alpha^2)$ and $R=\bZ+I$, an order in $K$ with conductor ideal $I$. Then $R$ is an order in a number field whose integral closure is an HFD and whose conductor ideal is $I=P^2$, with $P$ a non-principal prime $\Rbar$-ideal. Now, we can use the equivalent characterization of associated orders found in Theorem \ref{associated order} to show in finitely many steps that $R$ is in fact an associated order. Then by the previous theorem, $R$ must be an HFD.
    \end{example}

    Now recall by Theorem \ref{radical hfd ps} that if $R$ is a half-factorial order in a number field $K$ with radical conductor ideal $I$, $R[[x]]$ must be an HFD. The question still remains: does this result still hold if $I$ is non-radical? That is, is an order $R$ an HFD if and only if $R[[x]]$ is an HFD? Using Example \ref{non-radical hfd example}, we can answer this question in the negative. To start, we have the following lemmas.

    \begin{lemma}
        \label{power lies in R}
        Let $R$ be an ideal-preserving order in a number field $K$ with conductor ideal $I$. Then for any $\alpha\in\Rbar$, there exists some $k\in\bN$ such that $\alpha^k\in R$.
    \end{lemma}
    \begin{proof}
        Let $\alpha\in \Rbar$ and consider the ideal $J_1=\alpha\Rbar+I$, the smallest $\Rbar$-ideal containing both $\alpha$ and $I$. Let $J_2$ be the smallest ideal dividing $I$ which is relatively prime to $J_1$, i.e. $J_2$ is the product of all the primary ideals dividing $I$ which are relatively prime to $\alpha$. Then note that for some sufficiently large $m\in\bN$, $I|J_1^mJ_2$. Since $R$ is an ideal-preserving ideal, Theorems \ref{int orders} and \ref{intersect orders} tell us that $(R+J_1^m)\cap(R+J_2)=R$.

        Now note that since $\alpha$ is relatively prime to $J_2$, $\alpha\in U(\quot{\Rbar}{J_2})$, a finite group. Then there is some $j\in\bN$ such that $\alpha^j+J_2=1+J_2\in \quot{R+J_2}{J_2}$. Morevoer, for any $i\in\bN$, $\alpha^{ij}+J_2=1+J_2$, so $\alpha^{ij}\in R+J_2$ for every $i\in\bN$. Now let $k=ij$ for some $i\in\bN$ such that $k=ij\geq m$. Then $\alpha^k\in J_1^m\cap (R+J_2)\subseteq R$.
    \end{proof}

    \begin{lemma}
        \label{ps not an HFD}
        Let $R$ be an order in a number field $K$ with conductor ideal $I$. Suppose that there exists some $\Rbar$-ideal $J$ such that $I|J^2$ and $f\in\Irr(R[[x]])$ such that $f(x)=g(x)(\alpha+\beta x)$, with $g\in R[[x]]$ a nonunit, $\beta\in \Rbar$, and $\alpha\in J$. Then $R[[x]]$ is not an HFD.
    \end{lemma}
    \begin{proof}
        First, note that if $R$ is not an HFD, then it follows immediately that $R[[x]]$ is not an HFD from the atomicity of $R$ (see proof of Theorem \ref{integrally closed hfd iff ps}). Then assuming that $R$ is an HFD, we get that $R$ is an associated (and thus ideal-preserving) order. From the previous lemma, this tells us that there is some $k\in\bN$ such that $\beta^k\in R$ (and thus $\beta^{ik}\in R$ for any $i\in\bN$). Moreover, there is some $m\in \bN$ such that $m\alpha\in I$ (in particular, select $m\in J\cap \bZ$). Then $f^{mk}=g^{mk}(\alpha+\beta x)^{mk}$. On the left-hand side of this equality, we have a product of $mk$ irreducibles in $R[[x]]$. On the right-hand side, note that $g$ is a nonunit, so factoring $g$ into one or more irreducibles in $R[[x]]$ gives $g^{mk}$ written as a product of at least $mk$ irreducibles in $R[[x]]$. Finally, note that $$(\alpha+\beta x)^{mk}=\sum_{i=2}^{mk}\binom{mk}{i}\alpha^i(\beta x)^{mk-i}+k\beta^{mk-1}(m\alpha)x^{mk-1}+\beta^{mk}x^{mk}\in R[[x]].$$ Moreover, since $\alpha\in J$ is a nonunit in $\Rbar$, $(\alpha+\beta x)^{mk}$ is a nonunit in $R[[x]]$. Then $(\alpha+\beta x)^{mk}$ must factor into at least one irreducible in $R[[x]]$. Thus, the equation $f^{mk}=g^{mk}(\alpha+\beta x)^{mk}$ demonstrates a factorization in $R[[x]]$ with more irreducibles on the right than on the left, so $R[[x]]$ is not an HFD.
    \end{proof}

    \begin{lemma}
        \label{non-associated ps}
        Let $K$, $\alpha$, $I$, and $R$ be as in Example \ref{non-radical hfd example}. Then $\Rbar[[x]]\neq R[[x]]\cdot U(\Rbar[[x]])$; in particular, $3+\alpha x\notin R[[x]]\cdot U(\Rbar[[x]])$.
    \end{lemma}
    \begin{proof}
        Let $P=(3,2+2\alpha+\alpha^2)$, the prime $\Rbar$-ideal for which $I=P^2$, and let $R_1=R+P$. Note that since $R$ is an HFD, $R_1$ must be an associated order with (radical) conductor ideal $P$. Thus, $R_1$ is an HFD as well. From the fact that $\abs{\quot{\Rbar}{P}}=9$, $3\Rbar\subsetneq P$, and $2-4\alpha+\alpha^2\in I\subseteq P$, we observe that $\{\beta_1,\beta_2,\beta_3\}=\{1,\alpha,2-4\alpha+\alpha^2\}$ is an integral basis for $\Rbar$ such that $\{\beta_1,3\beta_2,\beta_3\}$ is an integral basis for $R_1$, $\{3\beta_1,3\beta_2,\beta_3\}$ is an integral basis for $P$, $\{\beta_1,9\beta_2,\beta_3\}$ is an integral basis for $R$, and $\{9\beta_1,9\beta_2,\beta_3\}$ is an integral basis for $I$. Moreover, note that $U(\Rbar)=\{\pm \alpha^k|k\in\bZ\}$, $U(R_1)= \{\pm \alpha^{4k}|k\in\bZ\}$, and $U(R)=\{\pm \alpha^{12k}|k\in\bZ\}$ (this can easily be verified by finding the smallest powers of $\alpha$ which lie in $R_1$ and $R$, respectively).

        Now suppose that there exists some $u=u_0+b_1x+b_2x^2+\dots\in U(\Rbar[[x]])$ such that $(3+\alpha x)u\in R[[x]]$. Then considering the constant term of this product, we have that $3u_0\in R$; from the integral bases above, it follows that $u_0\in U(R_1)$, i.e. $u_0=\pm \alpha^{4k}$ for some $k\in\bZ$. Now from the linear term of this product, we get that $3b_1+\alpha u_0=3b_1\pm \alpha^{4k+1}=r_1\in R\subseteq R_1$. Since $3\in P$, it follows that $3b_1\in R_1$. Thus, $\pm \alpha^{4k+1}=r_1-3b_1\in R_1$, a contradiction. Then no such $u\in U(\Rbar[[x]])$ can exist, so $(3+\alpha x)\notin R[[x]]\cdot U(\Rbar[[x]])$ and $\Rbar[[x]]\neq R[[x]]\cdot U(\Rbar[[x]])$.
    \end{proof}

    \begin{lemma}
        \label{irred reduces ps}
        Let $K$, $\alpha$, $I$, and $R$ be as in Example \ref{non-radical hfd example}, and let $R_1$ and $P$ be as in the proof of Lemma \ref{non-associated ps}. Then $f(x)=(6-12\alpha+3\alpha^2)+(1-2\alpha-4\alpha^2)x\in R[[x]]$ is an irreducible element of $R[[x]]$ which reduces in $\Rbar[[x]]$.
    \end{lemma}
    \begin{proof}
        In order to quickly verify when an element in question lies in $R$, $R_1$, $I$, or $P$, we will present elements in terms of the integral bases from the proof of Lemma \ref{non-associated ps}. Note that in $\Rbar[[x]]$, $f=3(2-4\alpha+\alpha^2)+(9-18\alpha-4(2-4\alpha+\alpha^2)x=(3+\alpha x)(2-4\alpha+\alpha^2)$, with neither $3+\alpha x$ nor $2-4\alpha+\alpha^2$ a unit (in fact, since their constant terms lie in $\Irr(\Rbar)$, both factors lie in $\Irr(\Rbar[[x]])$). Then $f$ reduces in $\Rbar[[x]]$. We must show that $f$ is irreducible in $R[[x]]$.

        Let $Q=(3,1+\alpha)$, the prime $\Rbar$-ideal such that $3\Rbar=PQ$, and note that the constant term of $f$ is a generator of $P^3Q$. Since neither $P$ nor $Q$ is a principal ideal, any factorization of $3(2-4\alpha+\alpha^2)$ must necessarily be the product of an generator of $P^2$ times a generator of $PQ$. Thus, if we were to factor $f$ into a product $f=gh$, with $g,h$ nonunit elements of $R[[x]]$, the constant terms of these factors must without loss of generality be $g_0=\pm 3\alpha^k$ and $h_0=\pm (2-4\alpha+\alpha^2)\alpha^{-k}$ for some $k\in\bZ$ such that $3\alpha^k\in R$ (since $2-4\alpha+\alpha^2\in I$, it will always multiply $\alpha^{-k}$ into $R$). As in the previous lemma, $k$ must therefore be a multiple of 4; we will replace $k$ with $4k$ in the following work to make this clear.

        Now considering the linear term of $f=gh$, we have $$f_1=9-18\alpha-4(2-4\alpha+\alpha^2)=\pm 3\alpha^{4k}h_1\pm (2-4\alpha+\alpha^2)\alpha^{-4k}g_1\in 3\alpha^{4k}R+(2-4\alpha+\alpha^2)\alpha^{-4k}R:=S_k$$ for some $k\in\bZ$. To show that such a factorization of $f$ is actually impossible, it will suffice to show that $9-18\alpha-4(2-4\alpha+\alpha^2)$ does not lie in $S_k$ for any $k\in\bZ$. Moreover, since $\alpha^{12}\in R$, it will suffice to show this only for $k\in\{-1,0,1\}$.

        Since any element of $R$ is of the form $a+9b+c(2-4\alpha+\alpha^2)$ for some $a,b,c\in \bZ$, we find the following:
        $$3R=\{3a+27b\alpha+3c(2-4\alpha+\alpha^2)\};$$
    	$$3\alpha^4R=\{(24a-162b+27c)+(-45a+540b-162c)\alpha+(-12a+27b+12c)(2-4\alpha+\alpha^2)\};$$
    	$$3\alpha^{-4}R=\{(402a+891b+432c)+(828a+1836b+891c)\alpha+(195a+432b+210c(2-4\alpha+\alpha^2))\};$$
    	$$(2-4\alpha+\alpha^2)R=\{(81b-36c)+(-162b+81c)\alpha+(a-36b+16c)(2-4\alpha+\alpha^2)\};$$
    	$$(2-4\alpha+\alpha^2)\alpha^4R=\{(9a+1296b-630c)+(-54a-2511b+1296c)\alpha+(4a-70b+289c)(2-4\alpha+\alpha^2)\};$$
    	$$(2-4\alpha+\alpha^2)\alpha^{-4}R=\{(144a+324b+153c)+(297a+648b+324c)\alpha+(70a+153b+76z)(2-4\alpha+\alpha^2)\}.$$
        Now suppose that $f_1\in S_0$. Then for some $a,b,c,x,y,z\in\bZ$,
        $$9=3a+18y-36z;$$
        $$-18=27b-162y+81z;$$
        $$-4=3c+x-36y+16z.$$
        Reducing the second equality modulo 27 gives $-18\equiv 0\modulo{27}$, a contradiction. Thus, $f_1\notin S_0$. Now suppose that $f_1\in S_1$. Then for some $a,b,c,x,y,z\in\bZ$,
        $$9=24a-162b+27c+144x+324y+153z;$$
        $$-18=-45a+540b-162c+297x+648y+324z;$$
        $$-4=-12a+27b+12c+70x+153y+76z.$$
        Reducing the first identity modulo 9 yields $24a\equiv 6a\equiv 0\modulo{9}\implies a\equiv 0\modulo{3}$. However, reducing the second identity modulo 27 gives $-45a\equiv 9a\equiv -18\modulo{27}\implies a\equiv 1\modulo{3}$, a contradiction. Then $f_1\notin S_1$. Finally, suppose that $f_1\in S_{-1}$. Then for some $a,b,c,x,y,z\in\bZ$,
        $$9=402a+891b+432c+9x+1296y-630z;$$
        $$-18=828a+1836b+891c-54x-2511y+1296z;$$
        $$-4 = 195a+432b+210c+4x-70y+289z.$$
        Reducing the first identity modulo 9 yields $402a\equiv 6a\equiv 0\modulo{9}\implies a\equiv 0\modulo{3}$. However, reducing the second identity modulo 27 gives $828a\equiv 18a\equiv -18\modulo{27}\implies a\equiv 2\modulo{3}$, a contradiction. Then $f_1\notin S_{-1}$, so $f$ must in fact be irreducible in $R[[x]]$.
    \end{proof}

    \begin{theorem}
        \label{non-hfd ps}
        Let $K$, $\alpha$, $I$, and $R$ be as in Example \ref{non-radical hfd example}, and let $R_1$ and $P$ be as in the proof of Lemma \ref{non-associated ps}. Then $R$ is a half-factorial order in a number field for which $R[[x]]$ is not an HFD.
    \end{theorem}
    \begin{proof}
        From Lemma \ref{irred reduces ps}, we have that $f(x)=(6-12\alpha+3\alpha^2)+(1-2\alpha-4\alpha^2)x$ is an irreducible element of $R[[x]]$ which factors in $\Rbar[[x]]$ as $f=(3+\alpha x)(2-4\alpha+\alpha^2)$. Note that $\alpha\in\Rbar$, $3\in P$, $I=P^2$, and $2-4\alpha+\alpha^2\in R[[x]]$ is a nonunit. Then it follows immediately from Lemma \ref{ps not an HFD} that $R[[x]]$ is not an HFD.
    \end{proof}

   \bibliographystyle{plain}
    \bibliography{bibliography}
\end{document}